\documentclass[12pt,reqno]{amsart}

\usepackage{xcolor}
\usepackage{lmodern}
\usepackage{mathrsfs}
\usepackage{tensor}
\usepackage[all]{xy}
\usepackage{amsfonts, amssymb,amsmath, amsthm, amsxtra, latexsym, amscd}
\usepackage{textcomp}
\usepackage[latin1]{inputenc}
\usepackage{graphicx}
\usepackage{bm}
\definecolor{hanblue}{HTML}{2B77A4}
\definecolor{hanblue}{rgb}{0.27, 0.42, 0.81}
\usepackage[colorlinks=true, citecolor=hanblue, linkcolor=hanblue, urlcolor=red]{hyperref}

\theoremstyle{definition}

\newtheorem*{cnj}{Geometric Arveson-Douglas Conjecture}

\newtheorem{thm}{Theorem}[section]
\newtheorem{cor}[thm]{Corollary}
\newtheorem{lem}[thm]{Lemma}
\newtheorem{prop}[thm]{Proposition}
\newtheorem{exm}[thm]{Example}
\newtheorem{rem}[thm]{Remark}
\newtheorem{defi}[thm]{Definition}
\numberwithin{equation}{section}

\makeatletter
\@namedef{subjclassname@2020}{\textup{2020} Mathematics Subject Classification}
\makeatother

\bigskip

\def\refn#1.#2{\expandafter\def\csname#1\endcsname{[#2]}}
\def\refnr#1.{\csname#1\endcsname}

\begin{document}
\baselineskip  1.2pc
\title[Essential normality of quotient submodules]{Essential normality of quotient submodules over strongly pseudoconvex finite manifolds}
\author[L. Ding]{Lijia Ding}
\address{School of Mathematics and Statistics,
Zhengzhou University,
Zhengzhou,
Henan 450001,
P. R. China}
\email{ljding@zzu.edu.cn}

\subjclass[2020]{Primary 46H25; Secondary 47A13; 32Q28; 19K56.}
\keywords{Hilbert module; Essential normality; Generalized Toeplitz operator; Embedded desingularization; Index theorem.}

\begin{abstract}
We investigate the  $p$-essential normality of  Hilbert quotient submodules on a relatively compact smooth strongly pseudoconvex domain in a complex manifold satisfying  Property (S). For analytic subvarieties that have compact singularities and transversely intersect the strongly pseudoconvex boundary, we prove that the corresponding Bergman-Sobolev quotient submodules are $p$-essentially normal whenever $p$ exceeds the dimension of the noncompact part of the analytic subvarieties. As a consequence, we partially confirm the geometric Arveson-Douglas Conjecture and resolve an open problem regarding the trace-class antisymmetric sum of truncated Toeplitz operators within a broader context. Moreover, we provide applications in $K$-homology and geometric invariant theory. 
\end{abstract}
\maketitle
\section{Introduction}
The study of $p$-essentially normal Hilbert modules began with Arveson's seminal work \cite{Arv98}. Arveson applied the results on $p$-essential normality to the geometric invariant theory of multioperators and obtained the operator-theoretic Gauss-Bonnet-Chern formula \cite{Arv02}, which relates the curvature invariant and the Fredholm index of the Dirac operator for an operator tuple. The research of this formula forms a crucial part of Arveson's influential program \cite{Arv05} to establish the analogy with the Atiyah-Singer index theorem in multioperator theory. Douglas \cite{Dou06A} expanded upon the study on the essential normality of Hilbert modules by aiming to establish a new kind of index theorem for subvarieties with singularities in smooth strongly pseudoconvex domains. This led to an analytic generalization of the Grothendieck-Riemann-Roch theorem for varieties with boundaries and singularities \cite{DJTY18, DTY16}.

The Arveson-Douglas Conjecture, initially formulated by Arveson \cite{Arv05} and later refined by Douglas \cite{Dou06A},  is not only a central issue in the $p$-essential normality theory of Hilbert modules but also connects with several famous conjectures in different subjects, such as the Hodge conjecture in algebraic geometry and Arveson's hyperrigidity conjecture in noncommutative analysis. 
The classical geometric version of the Arveson-Douglas Conjecture states that the quotient submodule on the unit ball determined by a homogeneous complex algebraic variety (the affine cone over a projective algebraic variety)  is $p$-essentially normal whenever $p$ exceeds the  dimension of the algebraic variety. This conjecture has been proved in various special cases (see, for example, \cite{Arv05, DJTY18, EnE15, KS12}), however, the general conjecture remains unresolved. Meanwhile,  the following generalized geometric Arveson-Douglas Conjecture on strongly pseudoconvex domains in complex Euclidean spaces has also been increasingly considered in the past decade. An intriguing aspect of these conjectures is that an operator-theoretic property for a Hilbert module is solely determined by the  dimension of the corresponding analytic subvariety, a purely geometric notion.  
The  partially affirmative results (see, for example, \cite{Dou06A, DGW22, EnE15, UpW16, WaX19}) on the generalized geometric Arveson-Douglas Conjecture reveal the profound interplay among the operator algebra and operator theory, complex and algebraic geometry, and several complex variables.
\begin{cnj} Let  $\Omega\subset \mathbb{C}^d$ be a bounded smooth strongly pseudoconvex domain and $V$ be an analytic subvariety  in some neighborhood of $\overline{\Omega}$  which intersects transversally with the boundary $\partial \Omega,$ then the   Hardy or  (weighted)  Bergman quotient submodule $M_V^\bot$ is $p$-essentially normal for all  $p>\text{dim}_{\mathbb{C}}\hspace{0.1em} (V\cap \Omega).$ 
\end{cnj} 

The purpose of this paper is to establish the $p$-essential normality theory in the context of complex manifolds,  which broadly generalizes and extends the $p$-essential normality theory over complex Euclidean spaces as mentioned above. More precisely, we address the problem for $p$-essential normality of the Bergman-Sobolev quotient submodules determined by analytic subvarieties on a relatively compact strongly pseudoconvex domain with the smooth boundary in a complex manifold with the Property (S), which can be regarded as the generalized geometric Arveson-Douglas Conjecture on a relatively compact strongly pseudoconvex domain in a complex manifold. Let $(W,w)$ be a Hermitian manifold of  dimension $d$ and $\Omega\Subset  W$ be a relatively compact strongly pseudoconvex domain with smooth boundary  $\partial \Omega.$  It is well-known from Grauert's theorem \cite{Gra58, Gra62} that  $\Omega$ is a proper modification \cite{CM03, GR84} of a Stein space. Moreover, $\Omega$ is a Stein manifold if and only if it has no positive dimensional compact analytic subvarieties. Following Kohn and Rossi \cite{KoR65}, the pair $(W,\Omega)$ is called a strongly pseudoconvex finite manifold. 

 \begin{defi}   A strongly pseudoconvex finite manifold $(W,\Omega)$ is said to have Property (S) if there exists a proper modification  $\iota: W\rightarrow\breve{W}$ for a Stein manifold  $\breve{W}$ along finitely many points such that $\iota$ is biholomorphic on a neighborhood of $\partial \Omega.$
\end{defi}

As an example, let $\breve{\Omega}\subset \breve{W}=\mathbb{C}^d$ be a bounded smooth strongly pseudoconvex domain, then the blowing-ups along finitely many points inside $\breve{\Omega}$ provide examples of a strongly pseudoconvex finite manifold with Property (S); particularly, $(\breve{W},\breve{\Omega})$ has Property (S). There may not exist a global holomorphic coordinate for the complex manifold $\Omega,$ thus the original definition of the  Hilbert module by the coordinate multipliers over the polynomial ring will no longer work in the general setting. For this reason, inspired by \cite{ Din24, Dou06A} we introduce the following relatively intrinsic definition of the Hilbert module over the holomorphic function algebra $\mathcal{O}(\overline{\Omega}),$ where $\mathcal{O}(\overline{\Omega})$ is the collections of holomorphic functions on some neighborhood of $\overline{\Omega}.$ The Grauert theorem mentioned above ensures that $\mathcal{O}(\overline{\Omega})$ is nontrivial. In fact, ${\mathcal{O}(\overline{\Omega})}$ is the inductive limit of the sets of sections of the holomorphic germ sheaf on $W$ over the open sets that contain $\overline{\Omega}.$  
 \begin{defi}\label{def} 
  A Hilbert space $H$  is called a Hilbert module over the holomorphic function algebra $\mathcal{O}(\overline{\Omega}),$ if for every $f\in \mathcal{O}(\overline{\Omega})$ there exists a bounded linear operator $M_f\in B(H)$ such that the map  $$\mathcal{O}(\overline{\Omega})\rightarrow B(H),\quad f\longmapsto M_f$$ is a representation of the algebra $\mathcal{O}(\overline{\Omega})$ on $H.$
 \end{defi}
 
Notice that $\overline{\Omega}\subset W$ is compact,  which implies that the Bergman-Sobolev space $\mathcal{O}^s(\overline{\Omega})$ is a Hilbert module in the Definition \ref{def} if each $M_f$ is chosen to be the multiplicative operator with the symbol $f\in\mathcal{O}(\overline{\Omega})$ for a given real number $s\in\mathbb{R}.$ Similarly, we can define the notion of Hilbert submodule and Hilbert quotient submodule.   When $\Omega$  is a bounded smooth strongly pseudoconvex domain in $\mathbb{C}^d,$  the Bergman-Sobolev spaces $\mathcal{O}^s(\overline{\Omega})$ with the canonical norm and the usual weighted Bergman spaces coincide if $s<{1}/{2}.$ In this case, we see that the polynomial ring $\mathbb{C}[\bm{z}]\subsetneq \mathcal{O}(\overline{\Omega}),$ thus every Hilbert module over the algebra $\mathcal{O}(\overline{\Omega})$ must be a Hilbert module over the polynomial algebra $\mathbb{C}[\bm{z}].$ Moreover,  when $\Omega=\mathbb{B}^d$  the author proved in  \cite{Din24} that a closed subspace in the Hardy space or weighted Bergman space is a submodule (resp. quotient submodule) over the polynomial algebra $\mathbb{C}[\bm{z}]$ if and only if it is a submodule (resp. quotient submodule) over the algebra $\mathcal{O}(\overline{\mathbb{B}^d}),$  namely this definition coincides with the original definition on the unit ball. 
 Correspondingly, the  $p$-essential normality of a Hilbert module over the algebra $\mathcal{O}(\overline{\Omega})$ should be modified and generalized. Recall from \cite{HeH75} that the 
antisymmetric sum  for given operators $A_1, \cdots, A_l$ on a Hilbert space $H$ is defined as
$$[A_1,\cdots ,A_l] = \sum_{\sigma\in S_l} \text{sgn}(\sigma)A_{\sigma(1)}\cdots A_{\sigma(l)},$$ where $S_l$ is the symmetric group on $\{1,  \cdots, l\}$ and $\text{sgn}(\sigma)$ is the signum of $\sigma.$
 
\begin{defi}\label{def1} Let $H$ be a Hilbert module  as in Definition \ref{def},   $H$ is said to be  $p$-essentially normal of length $l\in \mathbb{N}_+,$ denoted by $(p;l)$-essentially normal,   if all antisymmetric sums of the form \begin{equation}\label{come} [M_{f_1},M_{g_1}^*,\cdots,M_{f_l},M_{g_l}^*]\end{equation} belong to the Schatten $p$-class $\mathcal{L}^p$ for arbitrary  $f_1,\cdots,f_l,g_1,\cdots,g_l\in  \mathcal{O}(\overline{\Omega}),$ where $ 0< p\leq \infty.$  
\end{defi}
For the submodule and quotient submodule, each operator in (\ref{come}) should be replaced by its restriction and compression, respectively. We point out that a weighted Bergman module is $p$-essential normal in the sense of  Arveson  \cite{Arv05} if and only if it is $(p;1)$-essential normal in the above Definition \ref{def1} when  $\Omega=\mathbb{B}^d;$   the similar is also true for weighted Bergman submodules and quotient submodules \cite{Din24}. Therefore, $(p;1)$-essential normal will also be referred to as the traditional term $p$-essential normal throughout the paper. 

 To express our results clearly, we introduce some notations.  Let $\Omega, W$  as above,  and $V$ be an analytic subvariety in some relatively compact open neighborhood $\widetilde{\Omega}$ of $\overline{\Omega}$ such that $V_\Omega={V}\cap \Omega\neq\emptyset,$ whose singularities are compact and contained in $\Omega.$ Denote $V^+$ by the union of the noncompact irreducible components of $V,$ which is locally a finite union. Clearly, $V^+\subset \widetilde{\Omega}$ is  also a subvariety  and $V^+\subset V,$ moreover,  $(V\cap \Omega)^+\supset V^+\cap \Omega.$ We say that $V$  transversely intersects the  smooth boundary $\partial \Omega$  if $(V\cap \Omega)^+=\emptyset ,$ or $ (V\cap \Omega)^+ \neq\emptyset$ and its regular part $\text{Reg}\hspace{0.1em} (V)$ transversely intersects $\partial \Omega.$ We allow that ${V}$ has compact irreducible components inside  $\Omega.$  Let  $M_V^\bot\subset \mathcal{O}^{s}(\overline{\Omega})$ be the  Bergman-Sobolev quotient submodule determined by analytic subvariety $V$ for some $s\in\mathbb{R},$ where $M_V=\{f\in \mathcal{O}^{s}(\overline{\Omega}): f\vert_{V\cap\Omega}=\bm{0}\}$ is the submodule determined by $V.$ Observe that $M_V$ depends on $V\cap\Omega$, and if necessary, we shall write $M_{V\cap\Omega}$ instead of $M_V.$ The operator $S_f$ represents the compression of the multiplicative operator $M_f$ with symbol $f\in \mathcal{O}(\overline{\Omega})$,   also known as the truncated Toeplitz operator on $M_V^\bot$ with symbol $f$. Our main result is stated in the following theorem.
  \begin{thm}\label{main} Let $(W,\Omega)$ be a strongly pseudoconvex finite manifold with Property (S)  and $V$ be an analytic subvariety in some neighborhood of $\overline{\Omega}.$  If $V$ is smooth on $\partial \Omega$  and  intersects transversally with  $\partial \Omega,$  then the following hold for $s\in\mathbb{R}$:

(1) The  Bergman-Sobolev quotient submodule  $M_V^\bot\subset\mathcal{O}^{s}(\overline{\Omega})$ with the canonical norm is $(p;1)$-essentially normal for all   \begin{equation}\label{pes}
p>
\begin{cases}
\text{dim}_{\mathbb{C}}\hspace{0.1em} (V\cap \Omega)^+,&  (V\cap \Omega)^+ \neq\emptyset;\\ \notag
0,&  (V\cap \Omega)^+=\emptyset.
\end{cases}
\end{equation}

(2) The Bergman-Sobolev  quotient submodule  $M_V^\bot\subset\mathcal{O}^{s}(\overline{\Omega})$ with the canonical norm is  $(p;l)$-essentially normal with $l\geq2$ for all   \begin{equation}\label{pes}
p>
\begin{cases} \frac{1}{l+1} \hspace{0.1em} {\text{dim}_{\mathbb{C}}\hspace{0.1em} (V\cap \Omega)^+},
&  (V\cap \Omega)^+ \neq\emptyset;\\ \notag
0,&  (V\cap \Omega)^+=\emptyset. 
\end{cases}
\end{equation}
Moreover, for arbitrary  $f_1,\cdots,f_l,g_1,\cdots,g_l\in  \mathcal{O}(\overline{\Omega}),$ the antisymmetric sum $$[S_{f_1},S_{g_1}^*,\cdots,S_{f_l},S_{g_l}^*]\in \mathcal{L}^1$$  if $l\geq \text{dim}_{\mathbb{C}}\hspace{0.1em} (V\cap \Omega)^+,$ and its trace vanishes whenever $l>\text{dim}_{\mathbb{C}}\hspace{0.1em} (V\cap \Omega)^+.$
\end{thm}
We remark that if $V\cap \Omega\subset\Omega$  is a compact analytic subvariety of positive dimension then the above theorem provides a more precise result than the prediction of geometric Arveson-Douglas Conjecture,  i.e., the  Bergman-Sobolev quotient submodule  $M_V^\bot$ is  $p$-essentially normal for all $p>0.$ This is a new phenomenon in the general case of strongly pseudoconvex finite manifolds.
  When $W$ is a Stein manifold,  the noncompact part $(V\cap \Omega)^+$ is never empty unless $V\cap \Omega$ is zero-dimensional, and $\text{dim}_{\mathbb{C}}\hspace{0.1em} (V\cap \Omega)^+=\text{dim}_{\mathbb{C}}\hspace{0.1em} (V\cap \Omega)$ if  $(V\cap \Omega)^+$ is nonempty. For the crucial case where $W=\mathbb{C}^d,$ we have the following corollary which gives a partial affirmative answer for the aforementioned geometric Arveson-Douglas Conjecture. 
 
 \begin{cor}\label{corA} Let  $\Omega\subset\mathbb{C}^d$  be a bounded smooth  strongly pseudoconvex domain   and $V$ be an  analytic  subvariety  in some neighborhood of $\overline{\Omega}.$ If $V$ is smooth on $\partial \Omega$  and  intersects transversally with  $\partial \Omega,$ 
 then the following hold for $s\in\mathbb{R}$:

(1) The  Bergman-Sobolev quotient submodule  $M_V^\bot\subset\mathcal{O}^{s}(\overline{\Omega})$ with the  canonical norm  is  $(p;1)$-essentially normal for all   $p>\text{dim}_{\mathbb{C}}\hspace{0.1em} (V\cap \Omega).$

(2) The  Bergman-Sobolev quotient submodule  $M_V^\bot\subset\mathcal{O}^{s}(\overline{\Omega})$ with the canonical norm is $(p;l)$-essentially normal with $l\geq2$ for all   $p> \frac{1}{l+1} \hspace{0.1em} {\text{dim}_{\mathbb{C}}\hspace{0.1em} (V\cap \Omega)}.$ Moreover, for arbitrary  $f_1,\cdots,f_l,g_1,\cdots,g_l\in  \mathcal{O}(\overline{\Omega}),$ the antisymmetric sum $$[S_{f_1},S_{g_1}^*,\cdots,S_{f_l},S_{g_l}^*] \in \mathcal{L}^1$$ if $l\geq \text{dim}_{\mathbb{C}}\hspace{0.1em} (V\cap \Omega),$ and its trace vanishes whenever $l>\text{dim}_{\mathbb{C}}\hspace{0.1em} (V\cap \Omega).$
\end{cor}
 
Note that the Bergman-Sobolev space $\mathcal{O}^s(\overline{\Omega})$ with canonical norm is unitary equivalent to the Hardy space $\mathcal{O}^{-{1}/{2}+s}(\partial\Omega) $ for every $s\in\mathbb{R},$ and  $\mathcal{O}^{{d}/{2}}(\overline{\mathbb{B}^d}) $ with canonical norm coincides with the Drury-Arveson space, see Lemma \ref{phda}. 
 Consequently, Corollary \ref{corA} (1)  extends many previously known results in \cite{DGW22,EnE15,GuWk08,WaX19} on geometric Arveson-Douglas Conjecture. Moreover, Corollary \ref{corA} (2) resolves an open problem concerning the trace-class antisymmetric sum of truncated Toeplitz operators in \cite[Problem 12.3]{WaX23} recently proposed by Wang and Xia; our result is more general and sharp. We also give immediate applications of the main theorem in the  $K$-homology theory (see Proposition \ref{Kth}) and geometric invariant theory (see Proposition \ref{ginv}), which can be viewed as two specific examples of Atiyah-Singer index theorem in the multioperator theory. Notably, Proposition \ref{ginv} extends \cite[Proposition 7.10]{Arv00} and gives a positive answer to \cite[Problem D]{Arv07} for the compression of the $d$-shift of rank $1.$ 

 The main technical tools we used are the theory of generalized Toeplitz operators developed by Boutet de Monvel and Guillemin \cite{Bou79, BG81} and  Hironaka's embedded desingularization for analytic subvarieties \cite{AHV18, Hir64}. A key ingredient of the proof is the fact that the  $p$-essential normality of Bergman-Sobolev quotient submodules with canonical norms is preserved invariant under permissible blowing-ups, which can be regarded as an extension of the biholomorphic invariance of $p$-essential normality of quotient submodules \cite{Din24}.  The techniques developed here seem to have wide applicability and may shed new light on the geometric Arveson-Douglas Conjecture in the more general context and other related problems, for instance, the index theorem in the multioperator theory.  

The paper is organized as follows. In Section 2, we review some of the standard facts on complex algebraic geometry and the theory of generalized Toeplitz operators.  In Section 3, we study the properties of restriction operators.  Section 4 is devoted to proving the main theorem. Section 5 gives applications in the  $K$-homology theory and geometric invariant theory of the commuting multioperator. 

 \section{Preliminaries}
   \subsection{Analytic subvarieties and blowing-ups} All manifolds are assumed to be smooth, connected, Hausdorff,
and second countable.  Let $X$ be a complex manifold, a closed subset $A\subset X$ is called an analytic subvariety, if it is locally the common zeros of finitely many holomorphic functions. A point $x\in A$ is called smooth or regular, if there exists an open neighborhood $U\ni x$ such that $A\cap U$ is a closed complex submanifold of $U.$ Denoted $\text{Reg}\hspace{0.1em} (A)$ by the collections of all smooth points of $A,$ which is usually called the regular part of the subvariety $A.$  A point $x\in A$ is called singular, if $x\notin \text{Reg}\hspace{0.1em} (A),$ and the complement of the regular part denoted by $\text{Sing}\hspace{0.1em} (A)=A\setminus \text{Reg}\hspace{0.1em} (A) $ is usually called singularities or singular locus of $A.$ An analytic subvariety is called smooth if all points are smooth, or equivalently its singular locus is empty. It can be shown that $\text{Sing}\hspace{0.1em} (A)$ is also an analytic subvariety in $X$ and $\text{Sing}\hspace{0.1em} (A)$ is a thin (i.e., nowhere dense) subset in $A.$ Thus we can define the dimension of the analytic subvariety at a point $x$ by the dimension of complex manifolds \cite{Dem09, GR84}. The dimension of the empty set is conventionally defined to be $-1.$  We assume that the subvariety is always nonempty unless explicitly stated otherwise.  

Blowing-ups are a useful tool in complex and algebraic geometry,  especially in the theory of desingularization or resolution of singularities \cite{AHV18, Hir64}  over the complex number field $\mathbb{C},$ and we refer the reader to \cite{Kol07} for plentiful concrete examples of desingularization. Blowing-ups can be explicitly constructed in the local case and then extended globally via a gluing lemma \cite[Lemma 3.22]{Voi07}. Due to blowing-ups satisfying the universal property, we give the definition under the more general setting of complex space in the following theorem which is proven in \cite[Chapter 4.1]{Fis76}. Roughly speaking, the complex space is locally an analytic subvariety in an open set of a complex Euclidean space.

\begin{thm}\cite{Fis76} (Definition of blowing-up)    Let $X$ be a complex space with a closed complex subspace $Z\subset X$. Then there exists a holomorphic map
$\eta:\check{X} \rightarrow X$
with the following properties:
\begin{enumerate}
\item[(a)] $\eta$ is proper.
\item[(b)] $\eta^{-1}(Z) \subset \check{X} $ is a divisor (or hypersurface) with only normal crossing.
\item[(c)] $\eta$ is universal with respect to (b), i.e., if there is any holomorphic
map $\sigma:\tilde{X} \rightarrow X$ such that $\sigma^{-1}(Z)\subset \tilde{X}$ is a divisor then there is a unique holomorphic map $\phi: \tilde{X}\rightarrow \check{X}$ such that the diagram

\begin{displaymath} \xymatrix@R+1.3em@C+1.3em{\tilde{X} \ar[r]^<(.35){\phi}\ar[dr]_<(.35){\sigma}&\check{X}\ar[d]^{\eta}\\
&X}
\end{displaymath}

commutes.
\item[(d)] The restriction of $\eta: \check{X} \setminus \eta^{-1}(Z)\rightarrow X\setminus Z$
is biholomorphic.
\item[(e)]  If $X$ is a complex manifold and $Z$ is a complex submanifold, then $\check{X}$ is a complex manifold.
\end{enumerate}
$\eta:\check{X} \rightarrow X$ is called a blowing-up  (or monoidal transformation)
along center $Z,$ and $\eta^{-1}(Z) $ is called the exceptional divisor.
\end{thm}
The universal property (c) ensures that the blowing-up along a fixed center is unique up to biholomorphism. In general, a holomorphic map satisfies (a) and (d) is called a proper modification if $Z$ is chosen to be minimal. The following example gives the explicit construction of bowing-ups in the local case.
  \begin{exm} Let $X=\mathbb{C}^d,$ and $Z=\mathbb{C}^k\subset \mathbb{C}^d$ be a subspace with  $0\leq k\leq d-1.$ Denote $$\text{Bl}_{\mathbb{C}^k}(\mathbb{C}^d)=\{ (z,[w])\in \mathbb{C}^d\times \mathbb{P}^{\hspace{0.1em} d-k-1} :z_iw_j=z_jw_i, 1\leq i,j\leq d-k-1\},$$ where $\mathbb{P}^{\hspace{0.1em}l}$ denoted by $l$-dimensional  complex projective space. Then the projection $$\text{Pr}_1:\text{Bl}_{\mathbb{C}^k}(\mathbb{C}^d)\rightarrow \mathbb{C}^d,\quad (z,[w])\longmapsto z $$ is a blowing-up along the centre $\mathbb{C}^k.$ Clearly, when $k=d-1,$  the blowing-up of $\mathbb{C}^{d}$ along the centre $\mathbb{C}^{d-1}$ is the identity map $\text{Id}: \mathbb{C}^{d}\rightarrow \mathbb{C}^{d}.$ 
  \end{exm}
  
    \subsection{Bergman-Sobolev spaces}  For convenience, we first recall the definition of Sobolev spaces on a compact Riemannian manifold $(M,w_g)$ of dimension $d$ without boundary. The term $w_g$ is a $2$-form  determined uniquely by a complete Riemannian metric $g$ on $M$ and vice versa. Let $\mu$ be a (strictly) positive smooth $1$-density (viewed as a measure) on $M.$ The Sobolev space  $H^s(M,\mu)$  is defined as the completion of the space $C^\infty(M)$ with respect to the norm  $\Vert\cdot\Vert_s,$
    $$ \Vert f\Vert_s^2= \int_{M}\vert(\text{Id}-\Delta_g)^{\frac{s}{2}}f\vert^2\mu$$
   where $\Delta_g$ is the Laplace-Beltrami operator with respect to the metric $g.$  It is evident that  $H^0(M,\mu)= L^2(M,\mu).$  The compactness of $M$ implies that the definition of $H^s(M,\mu)$ is independent of the choice of $\mu,$  more precisely, these spaces are semi-equivalent. Two  Banach spaces $H_1, H_2$ are said to be semi-equivalent denoted by $ H_1\simeq H_2$ if they are equal as set and their norms are equivalent, this is nothing but a reformulation of the norm equivalence theorem. We will denote the Sobolev space $H^s(M,\mu)$ by $H^s(M)$ if $\mu$ is unambiguous in the context. 

Now we turn to the case where $\overline{\Omega}$ is a compact complex manifold with smooth strongly pseudoconvex boundary $\partial\Omega$ in a larger complex manifold $W.$  We can assume that  $\rho\in C^\infty(W)$ is a real-valued defining function of $\Omega$ such that
\begin{equation}\label{bdry} \overline{\Omega}=\{\rho\leq0\}, \quad\partial \Omega=\{\rho=0\},\quad \textbf{d}\rho|_{\partial  \Omega}\neq0\notag\\ \end{equation}
 and the Levi form $\mathcal{L}(\rho)$ is strictly positive at each point of $\partial  \Omega,$ the compactness of $\partial  \Omega$ implies that $\rho$ can be chosen  strongly plurisubharmonic on a neighborhood of $\partial \Omega.$  Let $\nu$ be the restriction to $\partial{\Omega}$ of the 1-form $\frac{1}{2\pi}\cdot\text{Im} (\partial\rho)=(\partial\rho-\bar{\partial}\rho)/(4\pi\sqrt{-1})$, the strict pseudoconvexity of $\partial\Omega$ implies that the $(2d-1)$-form $\lambda=\nu\land(\textbf{d} \nu)^{d-1}$ is positive and nowhere nonvanishing  on $\partial\Omega.$ Thus $\lambda$ is a smooth positive $1$-density.  We define $H^s(\partial{\Omega})$ as the Sobolev space $H^s(\partial{\Omega},\nu)$ for $s\in\mathbb{R}.$ We can choose an $\varepsilon>0$ small enough such that $\Omega_\varepsilon=\{\rho<\varepsilon\}\Subset W$ is smooth strongly pseudoconvex.   In what follows,  we take $M_\varepsilon=\overline{\Omega}_\varepsilon\cup_{\partial\Omega_\varepsilon}\overline{\Omega}_\varepsilon$ to be the double manifold of $\overline{\Omega}_\varepsilon, $ which contains $\Omega$ and is a compact real $2d$-dimensional  Riemannian manifold  without  boundary. There exists a Riemiannian metric $h_\varepsilon$ on $M_\varepsilon$ satisfying ${h}_\varepsilon|_\Omega=h|_\Omega.$ Note that $\overline{\Omega}\subsetneq\overline{\Omega}_\varepsilon \subsetneq M_\varepsilon,$ we denote the Sobolev space $H^s(\overline{\Omega})$ by $$H^s(\overline{\Omega}):=\{f|_{\overline{\Omega}}:f\in H^s(M_\varepsilon)\}, $$ where $f_{\overline{\Omega}}$ is the restriction of $f$ on $\overline{\Omega},$ the Sobolev norm $\Vert\cdot\Vert_{s,\overline{\Omega}}$ on $H^s(\overline{\Omega})$ is given by $$\Vert g\Vert_{s,\overline{\Omega}}=\inf \{ \Vert f\Vert_s: f|_{\overline{\Omega}}=g,f\in H^s(M_\varepsilon,w_{{h}_\varepsilon}^d)\}$$ for every $g\in H^s(\overline{\Omega}).$ It can be checked that $H^0(\overline{\Omega})=L^2(\Omega,w^d_{h_\varepsilon})=L^2(\Omega,w^d_{h}).$

Let $\textbf{d}$ be the exterior differential operator on the smooth sections (or forms) of complexified exterior algebra $\bigwedge{(T^\ast(W)\otimes \mathbb{C})}$ which can be decomposed into the sum of two differential operators $\partial:\mathcal{A}^{i,j}(W)\rightarrow\mathcal{A}^{i+1,j}(W)$ and $\bar{\partial}:\mathcal{A}^{i,j}(W)\rightarrow\mathcal{A}^{i,j+1}(W),0\leq i,j\leq d,$ namely $\textbf{d}=\partial+\bar{\partial}.$ The  Hermitian metric $w_h$ naturely induces  an inner product on each $\mathcal{A}^{i,j}(W),$ its completion is denoted by $L^2_{(i,j)}(W);$ for example, the inner product on $\mathcal{A}^{1,0}(W)$ is given by $$<f,t>_{}=\int_{W}<f,t>_{h,x}w_h^d=\sum_{l=1}^\infty\int_{U_l} \sum_{i,j}h^{i,\bar{j}}f_{U_l,i}\bar{t}_{U_l,j}\theta_lw^d_h,$$ where $(h^{i,\bar{j}})$ is the inverse matrix of $h$ and  $f=\sum_{i=1}^df_{U_l,i}dz_i,t=\sum_{i=1}^dt_{U_l,i}dz_i$ are local representations of  sections of $f,t\in\mathcal{A}^{1,0}(W),$ respectively.  Clearly, $L^2_{(0,0)}(W)=L^2(W).$ There exist operators $\star: \mathcal{A}^{i,j}(W)\rightarrow\mathcal{A}^{d-j,d-i}(W) $ called Hodge star operators such that $$<f,t>_{h,x}w_h^d=f\wedge\star\hspace{0.1em} \bar{t}$$ at each $x\in W.$ Locally, in the chart $(U,z),$ the Hodge star operator can be exactly given by \begin{equation}\label{Hog}\star f=\sqrt{\text{det}(h_{i,\bar{j}})}f_{L,\bar{K}}h^{(I^c\bar{L}^c)(\bar{K}L)}dz^{I\bar{J}}\end{equation}  for $f\in \mathcal{A}^{i,j}(W),$ where we temporarily adopt  notations in \cite{Koh63}.
Thus we have $$ <f,t>_{}=\int_Wf\wedge\star\hspace{0.1em} \bar{t}, $$ for all $f,t\in \mathcal{A}^{i,j}(W).$
Differential operators
$\partial,\bar{\partial}$ and Hodge star operators $\star$ can be uniquely extended to densely closed operators on $L^2_{(i,j)}(W),$ we denoted them also by $\partial,\bar{\partial}, \star $ respectively. 
 The dual operator of $\bar{\partial}$ under the above inner products is denoted by $\bar{\partial}^\ast.$ It can be checked that the following identities hold:
 \begin{equation}\begin{split}\label{phog}
\bar{\partial}^\ast&=-\star \partial\hspace{0.1em} \star. \\
\end{split}
 \end{equation}
 Let  $D_{\bar{\partial}}=-\bar{\partial}^\ast \bar{\partial}: L^2_{}(W)\rightarrow L^2_{}(W),$  then it can be checked   in the local chart $(U,z)$ by (\ref{Hog}) and (\ref{phog}) that $D_{\bar{\partial}}$ is a second order differential operator and its leading term is $$\sum_{i,j=1}^d h^{i,\bar{j}}\frac{\partial}{\partial {z}_i}\frac{\partial}{\partial \bar{z}_j},$$
which implies that $D_{\bar{\partial}}$ is   a second order  elliptic differential operator, in fact $$-D_{\bar{\partial}}=\Delta_{\bar{\partial}}$$  in the current case, where  $\Delta_{\bar{\partial}}=\bar{\partial}^\ast \bar{\partial}+\bar{\partial}\bar{\partial}^\ast $ is the  $\bar{\partial}$-Laplace-Beltrami operator. Usually, a function annihilated by $\Delta_{\bar{\partial}}$ is called harmonic.  

Now we consider the following elliptic boundary value problem of Dirichlet type:
\begin{equation}\label{ebva}
 \begin{cases}
D_{\bar{\partial}}f=0, \hspace{0.5em} \text{on} \hspace{0.3em} \Omega \\ 
f|_{\partial\Omega}=\phi, \hspace{0.4em} \phi\in C^\infty(\partial\Omega).
\end{cases}
 \end{equation}
 The solvability of the elliptic boundary value problem (\ref{ebva}) is guaranteed by \cite[Theorem VI. 1 (ii)]{See69} and the uniqueness is derived from the  Green-Stokes formula and the generalized idempotent property of Hodge star operators. Combined with  \cite[Theorem 5.2]{Win23} or \cite[Theorem VI. 1 (iii)]{See69}, it follows that there is a linear operator $K: C^\infty(\partial\Omega)\rightarrow C^\infty(\overline{\Omega})$  called Poisson extension operator that can be extended to a bijective continuous linear operator  $K: H^{-{1}/{2}+s}(\partial\Omega)\rightarrow \mathcal{N}(D_{\bar{\partial}},s) $ for each $s\in \mathbb{R},$  where $\mathcal{N}(D_{\bar{\partial}},s):=\{f\in H^s(\overline{\Omega}):  D_{\bar{\partial}}f=0 \hspace{0.3em}\text{on} \hspace{0.3em} \Omega\}$ which is a closed subspace of $H^s(\overline{\Omega})$ and is called harmonic-Sobolev space on $\overline{\Omega}.$ This implies that the definition of $\mathcal{N}(D_{\bar{\partial}},s)$ is independent of the choice of $\varepsilon,$  namely these spaces are semi-equivalent for different $\varepsilon.$
 By the Banach inverse mapping  theorem,  there exists a bijective continuous linear operator $\gamma:\mathcal{N}(D_{\bar{\partial}},s)\rightarrow H^{s-{1}/{2}}(\partial\Omega)$ such that \begin{equation} \label{krps} K\gamma=\text{Id}_{\mathcal{N}(D_{\bar{\partial}},s)}, \quad\gamma K=\text{Id}_{H^{s-{1}/{2}}(\partial\Omega)},\end{equation}  for each $s\in \mathbb{R},$ and $\gamma$ is the so-called 
  Sobolev trace operator. Thus the Poisson operator $K:  H^{s-{1}/{2}}(\partial\Omega) \rightarrow\mathcal{N}(D_{\bar{\partial}},s)$ will be a unitary operator if we chose suitable norms of $\mathcal{N}(D_{\bar{\partial}},s), H^{s-{1}/{2}}(\partial\Omega) $ for $s\in \mathbb{R},$ this will play a crucial role in our proof of the main theorem.  
 Note that $\mathcal{O}(\overline{\Omega})\subset\mathcal{N}(D_{\bar{\partial}},s)$ for each $s\in \mathbb{R},$ we define  $$\mathcal{O}^s(\overline{\Omega}):=\overline{\mathcal{O}(\overline{\Omega})\cap H^{s}(\overline{\Omega})}$$  to be  the  completion of $\mathcal{O}(\overline{\Omega})\cap H^{s}(\overline{\Omega})$ in  $ H^{s}(\overline{\Omega}),$ which is a closed subspace of $H^s(\overline{\Omega})$ and is called Bergman-Sobolev space on $\overline{\Omega}.$ Its image $\gamma(\mathcal{O}^{s+{1}/{2}}(\overline{\Omega})) $ under $\gamma$ is denoted by $$\mathcal{O}^s(\partial{\Omega}):=\gamma(\mathcal{O}^{s+{1}/{2}}(\overline{\Omega})) $$ which is called Hardy-Sobolev space on the boundary $\partial{\Omega}$ for each $s\in \mathbb{R}.$ According to  \cite{KoR65}, we know that a function $f\in C^\infty(\partial{\Omega})\cap H^s(\partial\Omega)$ belongs to $\mathcal{O}^s(\partial{\Omega})$ if and only if  $f$ satisfies the so-called tangental Cauchy-Riemann equations. Thus $\mathcal{O}^0(\partial{\Omega})$ is the usual Hardy space on $\partial{\Omega}.$ 
  
 \subsection{Generalized Toeplitz operators}  We first recall the definition of pseudodifferential operators on manifolds, however, we have to give the definition of pseudodifferential operators on an open set $U$ in the real Euclidean space $\mathbb{R}^d.$ Let $\mathscr{E}^\prime(U),\mathscr{D}^\prime(U)$ be the dual spaces of $C^\infty(U),$  the space of smooth functions, and $C^\infty_c(U),$ the space of smooth functions with compact supports, respectively.  A continuous linear operator $A:\mathscr{E}^\prime(U)\rightarrow \mathscr{D}^\prime(U)$ is called a pseudodifferential operator of order $m$ if there exists an amplitude function $a\in S^{m}(U)\subset C^\infty(U\times\mathbb{R}^d)$ such that $ A$ is the unique extension of the continuous operator $\text{Op}\hspace{0.1em} a: C^\infty_c(U)\rightarrow C^\infty(U)$ defined by \begin{equation}\label{pesd} (\text{Op}\hspace{0.1em}a) u(x)=\int_{\mathbb{R}^d} e^{\sqrt{-1}x\cdot \xi}a(x,\xi)\mathcal{F}(u)(\xi)d\xi,\quad x\in U, \end{equation} denoted by $A=\text{Op}\hspace{0.1em}a$ in this case.  Where $\mathcal{F}$ is the Fourier transform defined by \begin{equation}\label{fouri} \mathcal{F}(u)(\xi)=\frac{1}{(2\pi)^{d}}\int_{\mathbb{R}^d} e^{-\sqrt{-1}\xi\cdot y}u(y)dy,\quad u\in L^1(\mathbb{R}^d);\end{equation}
 $S^m(U)$ is  the H\"ormander symbol class space of order $\leq m,$ namely a function $a\in C^\infty(U\times \mathbb{R}^d)$ belongs to $S^m(U)$ if and only if, for any compact subset $Z\subset U$ and  multi-index of nonnegative integers $\alpha,\beta\in \mathbb{N}^d$, there exists a positive constant $C_{\alpha,\beta}(Z)$ such that 
 $$|\partial_x^\alpha\partial_\xi^\beta a(x,\xi)|\leq C_{\alpha,\beta}(Z)(1+|\xi|)^{m-|\beta|}$$ for all  $(x,\xi)\in Z\times\mathbb{R}^d,$ where  $\partial_\xi^{\beta}=(\frac{\partial }{\partial \xi_1})^{\beta_1}\cdots(\frac{\partial }{\partial \xi_d})^{\beta_d}$ and $|\beta|=\beta_1+\cdots+\beta_d$ as usual. Denote by $S^{-\infty}(U)=\cap_{m\in\mathbb{R}}S^m(U).$  Let $\Psi^m(U)$ be the collections of pseudodifferential operator of order $\leq m$ on $U,$  and  $\Psi^{-\infty}(U)=\cap_{m\in\mathbb{R}}\Psi^m(U).$ It can be proved that the formula (\ref{pesd}) introduces a linear bijection from $S^m(U)/S^{-\infty}(U)$ onto $\Psi^m(U)/\Psi^{-\infty}(U)$ for each $m\in \mathbb{R},$ and the equivalence class $[a]$ in $S^m(U)/S^{-\infty}(U)$ is called the formal symbol of  the pseudodifferential operator $A$ and denoted by $\sigma_{+}(A).$ We assume from now on that the pseudodifferential operator is always classical unless explicitly stated otherwise, namely for $a\in \sigma_{+}(A),$ which has an asymptotic expansion \begin{equation}\label{asym} a(x,\xi)\sim\sum_{j=0}^{+\infty}a_{j}(x,\xi) \end{equation}
 such that $a(x,\xi)-\sum_{j=0}^{N}a_{j}(x,\xi)\in S^{m-N}(U)$ for each $N\in\mathbb{N},$  where each $a_{j}(x,\xi)\in C^\infty(U\times\mathbb{R}^d)$ and is positive homogeneous of degree $m-j$ with respect to  $\xi$ for $\xi\neq\bm{0}.$
 It can be checked that the asymptotic expansion (\ref{asym}) is unique and is independent of the choice of $a.$ Thus we can define $\sum_{j=0}^{+\infty}a_{j}(x,\xi)$ to be the symbol of the (classical) pseudodifferential operator $A,$  and the first term $a_0(x,\xi)$ (or more generally, the equivalence class $[a] $ in $S^m(U)/S^{m-1}(U)$) is called the principal symbol and denoted by $\sigma(A).$ A  pseudodifferential operator is called elliptic if its principal symbol is nonvanishing on $U\times(\mathbb{R}^d\setminus{\bm{0}}).$ Suppose that $U^\prime \subset \mathbb{R}^d$ is an open set satisfying $\phi:U\rightarrow U^\prime$ is a diffeomoorphism, it introduces an isomorphism $\phi_\ast: \mathscr{D}^\prime(U)\rightarrow \mathscr{D}^\prime(U^\prime), $ especially $$\phi_\ast(f)(x) =f(\phi^{-1}(x))|\text{det }\phi^{-1}(x)|$$ for any locally integral function $f$ on $U.$ It is remarkable that  the pseudodifferential operator is invariant under 
the  diffeomorphism, namely, $A\in \Psi^m(U)$ if and only if its  transition $A^\phi \in \Psi^m(U^\prime),$ where $A^\phi = \phi_\ast\circ A\circ (\phi_\ast)^{-1}.$ 

Due to transition invariance, one can define the pseudodifferential operator on manifolds. Let $M$ be an arbitrary manifold of dimension $d.$ Since we have assumed that $M$ is smooth, connected, Hausdorff,
and second countable, it follows that $C^\infty(M)$ is a Fr\'echet space, and $C^\infty_c(M)$ is a  pre-Fr\'echet space and $C^\infty_c(M)$ is dense in  $C^\infty(M).$ The dual space of $C^\infty_c(M)$ is denoted by $\mathscr{D}^\prime(M)$ that is the space of distributions (identified with currents of degree $d$), and the dual space of $C^\infty(M)$ is denoted by $\mathscr{E}^\prime(M).$ A continuous linear operator $A:\mathscr{E}^\prime(M)\rightarrow \mathscr{D}^\prime(M)$ is called a pseudodifferential operator of order $m,$ if for every local chart $(U,\phi,x)$ of $M,$ then the compose map $A^\phi_U= ((\phi^{-1})_\ast)^{-1}\circ r_U\circ A \circ i_U\circ (\phi^{-1})_\ast$  is a  pseudodifferential operator of order $m$ on $\phi(U)\subset \mathbb{R}^d,$ where $i_u$ is the zero extension map and $r_U$ is the natural restriction map, see the following commutative diagram:

\begin{displaymath} \xymatrix@C+1.98em@R+.6em{ \mathscr{E}^\prime(\phi(U))\ar[r]^{(\phi^{-1})_\ast} \ar[d]_{A^\phi_U}&\mathscr{E}^\prime(U)\ar[r]^{i_U}&\mathscr{E}^\prime(M)\ar[d]^{A}\\\mathscr{D}^\prime(\phi(U)) &C^\infty_c(\phi(U))\ar[l]^{((\phi^{-1})_\ast)^{-1}}   &\mathscr{D}^\prime(M)\ar[l]^<(.27){r_U} .\\}
\end{displaymath}
The concept of principal symbol $\sigma(A)$ for a pseudodifferential operator $ A$ on the manifold $M$ can be globally defined on the cotangent bundle removed its zero section $T^\ast(M)\!\setminus\!\bm{0}.$
 Let $\Psi^m(M)$ be the collections of pseudodifferential operator of order $\leq m$ on $M.$  If $M$ is a compact Riemannian manifold, then for every $A\in \Psi^m(M),$ it can be extended to a continuous linear operator $A:H^s(M)\rightarrow H^{s-m}(M)$ for all $s\in \mathbb{R}.$
 A  pseudodifferential operator $A$ on $M$ is called elliptic if it is locally elliptic, namely its transition $A_U^\phi$  is elliptic for every local chart $(U,\phi,x)$ of $M,$ equivalently its principal symbol $\sigma(A)$  is nonvanishing on $T^\ast(M)\!\setminus\!\bm{0}.$  Pseudodifferential operators are special cases of (classical) Fourier integral operators which are introduced by H\"ormander, and the reader is also referred to \cite{EnE15, Tay81} for definition and more details.
 
 Now we turn to the case that $\overline{\Omega}$ is a compact complex manifold with smooth strongly pseudoconvex boundary $\partial\Omega$ in a larger complex $W,$ i.e., $\partial\Omega$ is a compact Riemannian manifold (without boundary) with the natural volume $\nu$ induced by $\rho$.
Let $\varPi: L^2(\partial{\Omega})\rightarrow \mathcal{O}^0(\partial{\Omega})$ be the  Szeg\"o orthogonal projection.
\begin{defi} Let $Q\in \Psi^m(\partial \Omega)$ for some $m\in\mathbb{R}.$ Then the generalized  Toeplitz operator $T_Q: \mathcal{O}^s(\partial{\Omega})\rightarrow \mathcal{O}^{s-m}(\partial{\Omega})$ is given by $$T_Q= \varPi Q \varPi$$ for any $s\in\mathbb{R}.$
\end{defi}
This is the natural extension of classical Toeplitz operators on the unit circle $\partial \mathbb{B}^1.$  The generalized  Toeplitz operators are all continuous, since the Szeg\"o  orthogonal projection $\varPi: \mathcal{O}^s(\partial{\Omega})\rightarrow \mathcal{O}^{s}(\partial{\Omega})$ is continuous for all $s\in\mathbb{R}.$  The microlocal structure of generalized Toeplitz operators was described by Boutet de Monvel and Guillemin \cite{Bou79, BG81}.

 \section{Properties of restriction operators}\label{ser}
 Global decomposition theorem \cite[Therem II.5.3]{Dem09} states that every analytic subvariety can be uniquely decomposed into the locally finite union of global irreducible components,  and each global irreducible component is an analytic subvariety and is the topological closure of irreducible components of the regular part of the analytic subvariety.  
 Consequently, we can suppose first that the analytic subvariety $V$  of the relatively compact open  neighborhood $\widetilde{\Omega}$ of $\overline{\Omega}$ can be expressed as a union of at most finitely many and mutually different  complex submanifolds  $V_1,\cdots, V_m$  of $\widetilde{\Omega}$ given by
  $$V=\bigcup_{i=1}^m V_i,$$ 
where  $V_i \cap \Omega\neq\emptyset$ for $1\leq i\leq m.$ Indeed, each complex submanifold $V_i$ is a smooth irreducible component for $1\leq i\leq m.$  In this section, we will focus on the case where all smooth irreducible components $V_i$ are noncompact, while the compact case will be addressed in the subsequent section.
  
Denote
 $k_i=d-\text{dim}_{\mathbb{C}}\hspace{0.1em}V_i $ by the codimension of the complex submanifold $V_{i,\Omega}\subset \Omega,$ where $V_{i,\Omega}= V_i\cap \Omega,1\leq i\leq m.$ Let $$ R_{V_{i,\Omega}}: \mathcal{O}^s(\overline\Omega)\rightarrow  \mathcal{O}^{s-{k_i}/2}(\overline{V_{i,\Omega}}),\quad f\longmapsto f|_{\overline{V_{i,\Omega}}}$$ for all $i=1,\cdots,m,$ more generally we define the restriction map $R_V$ between two direct sum spaces as 
$$R_{V_{\Omega}}: \bigoplus_{i=1}^m \mathcal{O}^s(\overline\Omega)\rightarrow \bigoplus_{i=1}^m \mathcal{O}^{s-k_i/2}(\overline{V_{i,\Omega}}),\quad (f_1,\cdots,f_m)\longmapsto (R_{V_{1,\Omega}}f_1,\cdots,R_{V_{m,\Omega}}f_m).$$ Where and below we identify the direct sum and the direct product of finitely many Hilbert spaces with the natural definitions of inner products. 

Recall that a continuous operator between two Hilbert spaces is said to be right semi-Fredholm if its range is closed and has finite codimension and a continuous operator on a Hilbert space is right semi-Fredholm if and only if its equivalence class in Calkin algebra is right invertible by \cite[Proposition XI.2.5]{Coy90}.  Similarly,  a continuous operator between two Hilbert spaces is said to be left semi-Fredholm if its range is closed and its kernel has a finite dimension. Clearly, a continuous operator between two Hilbert spaces is left semi-Fredholm if and only if its adjoint operator is right semi-Fredholm.  Observe that a continuous operator between two Hilbert spaces is Fredholm in the usual sense if and only if it is both right semi-Fredholm and left semi-Fredholm.
\begin{lem} The restriction map $R_{V_\Omega}$ is right semi-Fredholm.
\end{lem}
\begin{proof}  It follows from \cite[Corollary 3]{EnE15} that each restriction map $R_{V_{i,\Omega}}$ is right semi-Fredholm, i.e., each $R_{V_{i,\Omega}}$ is continuous and its range $\text{Ran}\hspace{0.1em} R_{V_{i,\Omega}}$ has finite codimension, which implies that $\text{Ran}\hspace{0.1em} R_{V_{i,\Omega}}$ is closed in $ \mathcal{O}^{s-k_i/2}(\overline{V_{i,\Omega}})$ and $$\text{dim}\hspace{0.1em} \mathcal{O}^{s-k_i/2}(\overline{V_{i,\Omega}})/\text{Ran}\hspace{0.1em} R_{V_{i,\Omega}}<\infty$$ for all $i=1,\cdots,m.$ Since $\text{Ran}\hspace{0.1em} R_{V_{\Omega}}= \bigoplus_{i=1}^m \text{Ran}\hspace{0.1em} R_{V_{i,\Omega}}$ is closed and the following two Hilbert spaces $$\bigoplus_{i=1}^m  \mathcal{O}^{s-k_i/2}(\overline{V_{i,\Omega}})/\bigoplus_{i=1}^m \text{Ran}\hspace{0.1em} R_{V_{i,\Omega}}\hspace{0.5em}\text{and}\hspace{0.5em}\bigoplus_{i=1}^m  \mathcal{O}^{s-k_i/2}(\overline{V_{i,\Omega}})/ \text{Ran}\hspace{0.1em} R_{V_{i,\Omega}}$$ are unitary equivalent, it follows that

$$\text{dim}\hspace{0.1em} \bigoplus_{i=1}^m  \mathcal{O}^{s-k_i/2}(\overline{V_{i,\Omega}})/\text{Ran}\hspace{0.1em} R_{V_{\Omega}}=\sum_{i=1}^m \text{dim}\hspace{0.1em} \mathcal{O}^{s-k_i/2}(\overline{V_{i,\Omega}})/ \text{Ran}\hspace{0.1em} R_{V_{i,\Omega}}<\infty.$$ This completes the proof.
\end{proof}

\begin{rem}\label{KRV} When $V\subset\Omega$ is irreducible and smooth, namely $V\subset\Omega$ is a complex submanifold, it is clear that $M_V=\text{Ker}R_{V_{\Omega}}$ and $M_V^\bot=(\text{Ker}R_{V_{\Omega}})^\bot,$ thus  it introduces a nature injection from $M_V^\bot$ into $  \mathcal{O}^{s-k/2}(\overline{{V_{\Omega}}}),$ where $k=d-\text{dim}_{\mathbb{C}}\hspace{0.1em}V$ is the codimension of $V.$ Furthermore, if $\Omega$ is a Stein manifold, then the restriction map $R_{V_{\Omega}}$ is actually surjective for $s<1/2,$ thus it introduces an invertible continuous map from $M_V^\bot$ onto $  \mathcal{O}^{s-k/2}(\overline{{V_{\Omega}}})$ for $s<1/2.$
\end{rem}

Since $V$ is smooth in a neighborhood of $\partial\Omega$ and intersects $\partial\Omega$ transversally by the hypothesis, which implies that each of its irreducible components $V_i$ also intersects $\partial\Omega$ transversally and with the smooth strongly pseudoconvex boundary $ \partial(V_{i,\Omega})= \partial\Omega\cap V_i.$   For each irreducible component $V_i,$ we denote $\rho_{V_i}$ by its related defining function as well as its almost-analytic extension, by the transversality we can take $\rho_{V_i}:=\rho|_{V_i\times V_i}.$ We also denote by $\nu_{V_i}=\text{Im} (\partial\rho_{V_i})$ the corresponding $1$-form on the boundary $\partial\Omega\cap V_i,$ and by $\varPi_{V_i},{K}_{V_i},\gamma_{V_i},$ the Szeg\"{o} projection, Poisson operator, and trace operator, respectively. In particular, when $\overline{\Omega}\subset V_i,$ for some irreducible component $V_i,$ we know that $\overline{\Omega}\subset V=V_i,$ and
then $\Omega\cap V=\Omega\cap V_i=\Omega,$ thus in this case,  $\rho_{V_i}$ degenerates to $\rho$ on the whole of $\Omega,$ and $\eta_{V_i},\varPi_{V_i},{K}_{V_i},\gamma_{V_i}$ degenerate to $\eta,\varPi,{K},\gamma$ on the whole of $ \partial\Omega,$ respectively. We define the operator $R_{\partial V_i}$ as \begin{equation}\label{reci} R_{\partial V_i}:=\gamma_{V_i}R_{V_i}{K}: C^{\infty}( \partial\Omega)\rightarrow C^{\infty}( \partial\Omega\cap V_i),\end{equation} which is the restriction of $R_{V_i}$ to the boundary $\partial\Omega\cap V_i,$ $i=1,\cdots,m.$

We now collect some facts on generalized Toeplitz operators in the following proposition, which will be needed later on.
\begin{prop}\cite{Bou79,BG81,EnE15} The  generalized  Toeplitz operators possess the following properties:
\begin{enumerate}
\item[(P1)] For any $T_P$, $P\in\Psi^m(\partial \Omega)$, there in fact exists
$Q\in\Psi^m(\partial \Omega)$ such that $T_P=T_Q$ and $Q\varPi=\varPi Q$.
(Hence, $T_P=T_Q$ is just the restriction of $Q$ to the Hardy space. 
 Particularly, it follows that generalized Toeplitz operators
form an algebra.)
\item[(P2)]  It can happen that $T_P=T_Q$ for two different $P$ and $Q,$ with the same principal symbol restricted to the half-line bundle $ \Sigma.$  Moreover, one can define unambiguously the order of $T_Q$ as $$\text{ord}\hspace{0.1em}(T_Q):=\min\{\text{ord}\hspace{0.1em}(P):\;
T_P=T_Q, \sigma(P)|_\Sigma\neq\bm{0}\},$$ and the symbol of $T_Q$ as $\sigma(T_Q):=\sigma(Q)|_\Sigma$
if $\text{ord}\hspace{0.1em}(Q)=\text{ord}\hspace{0.1em}(T_Q)$.
\item[(P3)]  The order and the symbol of generalized  Toeplitz operators obey the usual laws: $\text{ord}\hspace{0.1em}(T_QT_P)=
\text{ord}\hspace{0.1em}(T_Q)+\text{ord}\hspace{0.1em}(T_P)$ and $\sigma(T_Q T_P)=\sigma(T_Q)\sigma(T_P)$.
\item[(P4)]  If $\text{ord}\hspace{0.1em}(T_P)=m$, then $T_P$ maps $\mathcal{O}^s(\partial{\Omega})$ continuously into
$\mathcal{O}^{s-m}(\partial{\Omega})$, for any $s\in\mathbb{R}$. In~particular, if $\text{ord}\hspace{0.1em}(T_P)=0$ then
$T_P$ is a bounded operator on~$L^2(\partial{\Omega})$; if~$\text{ord}\hspace{0.1em}(T_P)<0$, then it is
even compact. 
\item[(P5)]  If $P\in\Psi^m(\partial \Omega)$ and $\sigma(T_P)=0$, then there exists
$Q\in\Psi^{m-1}(\partial \Omega)$ with $T_Q=T_P$.

\item[(P6)]  A generalized Toeplitz operator $T_P$
is said to be elliptic if $\sigma(T_P)$ does not vanish.
Then $T_P$ has a parametrix, i.e., there exists an elliptic generalized Toeplitz
operator $T_Q$ of order $-\text{ord}\hspace{0.1em}(T_P)$, with $\sigma(T_Q)=\sigma(T_P)^{-1}$,
such that $T_P T_Q-I$ and $T_Q T_P-I$ are smoothing operators
(i.e.,~have Schwartz kernel in $C^\infty(\partial{\Omega}\times\partial{\Omega})$).
\end{enumerate}
\end{prop}
 Where $\Sigma$ in (P2) is denoted  by the half-line bundle 
$$ \Sigma := \{(x,\xi)\in T^*(\partial\Omega): \xi=t\nu_x, t>0 \}, $$
where $\nu$ is the restriction to $\partial{\Omega}$ of the 1-form $\frac{1}{2\pi}\cdot\text{Im} (\partial\rho)=(\partial\rho-\bar{\partial}\rho)/(4\pi\sqrt{-1})$; the strict pseudoconvexity of $\partial\Omega$ implies that $\Sigma$ is a symplectic submanifold of the cotangent bundle $T^*(\partial\Omega).$ 

Let $P$ be a positive elliptic pseudodifferential operator of order $2s\in \mathbb{R},$ denote $H_P$ by the completation of $\mathcal{O}(\partial\Omega)$ with the norm $$\Vert u\Vert_{H_P}=\sqrt{\langle P u,u\rangle_{\partial\Omega} },\quad u \in \mathcal{O}(\partial\Omega),$$ where  $\mathcal{O}(\partial\Omega)$ is the set of smooth functions satisfying the tangental Cauchy-Riemann equations. It follows from (P6) and the  Seeley calculus for elliptic pseudodifferential operators that  $\mathcal{O}^s(\partial\Omega)$ and $H_P$ are semi-equivalent.  Moreover, for the positive elliptic generalized Toeplitz $T_P$ of order $2s,$ it can be checked that \begin{equation}\label{hpht} H_P=H_{T_P}.\notag\\ \end{equation}   
 Note that the Poisson operator $K:  L^2(\partial\Omega) \rightarrow L^2(\Omega)$ is bounded, its adjoint is denoted by $K^\ast:  L^2(\Omega)\rightarrow L^2(\partial\Omega).$ It follows from \cite{Bou79} that the operator $\Lambda_{s}=K^\ast M_{\vert\rho\vert^s}K$ is a positive elliptic pseudodifferential operator of order $-s-1$ for all $s>-1.$ In particularly, $\Lambda_{0}$ has order $-1.$ By direct calculation, we have \begin{equation}\label{lkoe}\int_\Omega \vert Ku\vert^2\vert\rho\vert^{2s}w^d=\langle K^\ast M_{\vert\rho\vert^{2s}}Ku,u\rangle_{\partial\Omega}= \langle \Lambda_{2s} u,u\rangle_{\partial\Omega}.\end{equation}   For each $s>-1,$ let $$A^2_{s,\rho}(\Omega)=L^2(\Omega,|\rho|^s w^d)\bigcap \mathcal{O}(\overline{\Omega}).$$ It follows from (\ref{lkoe}) that $\mathcal{O}^{s}(\overline{\Omega})$ is semi-equivalent to the weighted Bergman space $A^2_{-2s,\rho}(\Omega)$ for $s<1/2.$ 
 
We now endow $\mathcal{O}^{s}(\overline{\Omega})$ with an equivalent norm for every $s\in\mathbb{R},$ which coincides with the usual norm when $\Omega$ is the unit ball.  Denote $\{F_s\}_{s< -1}$ by a family of smooth functions on $\mathbb{R}$ satisfying the following conditions:

(1) For   $-(d+1) <s < -1,$  $$F_s(x)=\frac{\Gamma(dx)}{\Gamma(d)\Gamma(dx+s+1)},\quad \forall x\geq 1-\frac{1}{2d};$$

(2) For   $ s\leq -(d+1),$   $$F_s(x)=\frac{\Gamma(dx)}{(dx-d+1)^{d+s}\Gamma(dx-d+1)},\quad \forall x\geq 1-\frac{1}{2d};$$

(3) $F_s(x)=0, \forall x\leq -1. $
 
 We define $B_s=F_s(\Lambda_0^{-1})$  by continuous function calculus for (unbounded) selfadjoint operators for $s<-1.$ By direct calculation, we see that $F_s$ belongs to  H\"ormander class $S^{-s-1}(\mathbb{R})$ for every $s<-1.$ Combined with \cite[Theorem XII.1.3]{Tay81} for functional calculus of pseudodifferential operators, we conclude that $B_s$ is a positive elliptic pseudodifferential operator of order $-s-1$ for every $s<-1.$ We now define a family $\{\Vert\cdot\Vert_{\partial\Omega,s} \}_{s\in \mathbb{R}}$ of norms on $\mathcal{O}(\partial\Omega)$ as follows:
 
 \begin{equation}\label{norsm} \Vert u\Vert_{\partial\Omega,s}^2= \begin{cases}
{\langle \Lambda_{-2s-1} u,u\rangle_{\partial\Omega}}, & s<0 ;\\   
 {\langle  u,u\rangle_{\partial\Omega}},&  s=0;\\
  {\langle B_{-2s-1} u,u\rangle_{\partial\Omega}},& s>0.\\
\end{cases}
\end{equation}
From the above, we see that $\Vert\cdot\Vert_{\partial\Omega,s}$ is an  equivalent norm of $\mathcal{O}^s(\partial\Omega)$ for every $s\in\mathbb{R}.$ Denote the norm $\Vert\cdot\Vert_{\Omega,s}$ on $\mathcal{O}^{s}(\overline{\Omega})$ as \begin{equation}\label{norms}\Vert u\Vert_{\Omega,s}=\Vert\gamma u\Vert_{\partial\Omega,s-1/2}\end{equation} for every $s\in\mathbb{R}.$ It follows from (\ref{krps}) that $\Vert\cdot\Vert_{\Omega,s}$ is an  equivalent norm of $\mathcal{O}^{s}(\overline{\Omega})$ for every $s\in\mathbb{R}.$ The norms defined in (\ref{norsm}) and (\ref{norms}) are called canonical norms. The canonical norms enjoy the following properties.

\begin{lem}\label{phda} (1) The Poisson operator \begin{equation}K: (\mathcal{O}^{s-1/2}(\partial\Omega),\Vert\cdot\Vert_{\partial\Omega,s-1/2})\rightarrow(\mathcal{O}^{s}(\overline{\Omega}),\Vert\cdot\Vert_{\Omega,s}) \notag\end{equation} and its inverse  $\gamma$ are  unitary for every $s\in\mathbb{R}.$

(2) $(\mathcal{O}^{s}(\overline{\Omega}),\Vert\cdot\Vert_{\Omega,s})=A^2_{-2s,\rho}(\Omega)$  for $s<1/2.$ 

Moreover, when $W=\mathbb{C}^d$ with the usual Hermitian mertic and $\Omega=\mathbb{B}^d$ with the defining function $\rho(z)=\vert z\vert^2-1,$ the following hold.

(3) $(\mathcal{O}^{{1}/{2}}(\overline{\mathbb{B}^d}),\Vert\cdot\Vert_{\Omega,1/2})=H^2(\mathbb{B}^d),$   the  Hardy space on $\mathbb{B}^d.$

(4)  $(\mathcal{O}^{{d}{/2}}(\overline{\mathbb{B}^d}),\Vert\cdot\Vert_{\Omega,d/2})=H^2_d(\mathbb{B}^d),$  the Drury-Arveson space on $\mathbb{B}^d.$
\end{lem}

\begin{proof} (1) It is immediate from (\ref{krps}) and (\ref{norms}).

(2)  It follows from (\ref{lkoe}) that $(\mathcal{O}^{s}(\overline{\Omega}),\Vert\cdot\Vert_{\Omega,s}) \simeq A^2_{-2s,\rho}(\Omega).$ Togather with  (\ref{norms}) leads to the desired equality of Hilbert spaces.

(3) By (\ref{norms}), it suffices to prove $(\mathcal{O}^{0}({\partial\mathbb{B}^d}),\Vert\cdot\Vert_{\partial\mathbb{B}^d,0})=H^2(\mathbb{B}^d)$. Since $\overline{\mathbb{B}^d}$ is compact, by  the Stokes formula we have \begin{equation}\begin{split}
\Vert z^n \Vert_{\partial\mathbb{B}^d,0}^2&=\int_{\partial\mathbb{B}^d}\vert z^n\vert^2 \lambda\\ \notag
&=\frac{d+\vert n\vert}{d}\int_{\mathbb{B}^d}\vert z^n\vert^2\frac{d!}{\pi^d}\prod_{j=1}^d(\frac{\sqrt{-1}}{2}\textbf{d}z_j\land\textbf{d}\bar{z}_j)\\
&=\frac{(d-1)!n!}{(d-1+\vert n\vert)!}
\end{split}\end{equation}
 for all $n,$ where $n=(n_1,\cdots,n_d)\in \mathbb{N}^d, z^n=z_1^{n_1}\cdots z_d^{n_d}, n!=\prod_{j=1}^d(n_j!), \vert n\vert =n_1+\cdots +n_d.$

(4) By direct calculation, we have  $$\Vert z^n \Vert_{\mathbb{B}^d,d/2}^2=\langle B_{-d} z^n,z^n\rangle_{\partial\mathbb{B}^d,0}=\frac{n!}{\vert n\vert!}$$ for all $n\in \mathbb{N}^d.$
\end{proof}

\begin{prop}\label{rpdu} The Bergman-Sobolev space $(\mathcal{O}^{s}(\overline{\Omega}),\Vert\cdot\Vert_{\Omega,s})$ is a  reproducing kernel Hilbert space for each $s\in\mathbb{R}.$ \end{prop}
\begin{proof}  For $s<1/2,$  it follows from (\ref{phda}) that  $(\mathcal{O}^{s}(\overline{\Omega}),\Vert\cdot\Vert_{\Omega,s})=A^2_{-2s,\rho}(\Omega).$  For any $z_0\in\Omega,$ there exists a local coordinate chart $(U,\phi,z)$ such that  $U\subset \mathbb{C}^d$ has finite volum $\text{Vol}(U)<\infty$ and 
\begin{equation}\begin{split}\label{evaf} \vert f(z_0)\vert&=\vert f\circ\phi^{-1}(\phi(z_0))\vert\\
&=\Big\vert\frac{1}{\text{Vol}(U)}\int_{U}f\circ\phi^{-1}(z)dz\Big\vert\\
&=\frac{1}{\text{Vol}(U)}\Big\vert\int_{U}f\circ\phi^{-1}(z)\vert \rho\circ\phi^{-1}\vert^{-s} \sqrt[4]{\det(g)} \frac{\vert \rho\circ\phi^{-1}\vert^{s}}{ \sqrt[4]{\det(g)} }dz\Big\vert\\
&\leq \frac{1}{\text{Vol}(U)} \Big(\int_{U}\vert f\circ\phi^{-1}(z)\vert^2\vert \rho\circ\phi^{-1}\vert^{-2s} \sqrt[2]{\det(g)} dz \Big)^{\frac{1}{2}}\Big(\int_{U} \frac{\vert \rho\circ\phi^{-1}\vert^{2s}}{ \sqrt[2]{\det(g)} } dz \Big)^{\frac{1}{2}}\\
&\leq C \Vert f\Vert_{A^2_{-2s,\rho}(\Omega)}.
\end{split}\end{equation}
This shows that the evaluation functional is bounded, and hence  $(\mathcal{O}^{s}(\overline{\Omega}),\Vert\cdot\Vert_{\Omega,s}) $ is a reproducing kernel Hilbert space for each  $s<1/2.$  It follows from (\ref{norsm}) and (\ref{norms}) that the trace operator $\gamma: \mathcal{O}^s(\overline{\Omega})\rightarrow\mathcal{O}^{s-1/2}(\partial\Omega)$ is unitary with canonical norms for every $s\in\mathbb{R}.$  In particular, the Bergman-Sobolev space $\mathcal{O}^{-1/2}(\overline{\Omega})$ is unitary to $\mathcal{O}^{-1}(\partial\Omega).$ By Rellich's lemma for Sobolev spaces on compact Riemannian manifolds, it follows that the embedding $\text{Id}:  \mathcal{O}^{s-1/2}(\partial\Omega)\rightarrow\mathcal{O}^{-1}(\partial\Omega)$ is compact for each $s\geq 1/2.$ Combined with (\ref{evaf}), it follows that the evaluation functional is bounded on $\mathcal{O}^{s-1/2}(\partial\Omega)$ for $s\geq 1/2.$ Then $\mathcal{O}^{s-1/2}(\partial\Omega)$  is a reproducing kernel Hilbert space, and so is   $\mathcal{O}^s(\overline{\Omega})$ for each $s\geq 1/2.$ This completes the proof.
\end{proof}
From now on the norms of the Bergman-Sobolev space $\mathcal{O}^{s}(\overline{\Omega})$ and  Hardy-Sobolev space are assumed to be the canonical norms as defined above with respect to a given Hermitian metric unless otherwise stated.   The example provided in \cite[Remark 4.7]{Arv07} demonstrates that 
the $p$-essential normality is not preserved in general when the norm is instead by an equivalent norm. 

 We have the following key lemma by Weyl's law for positive elliptic generalized Toeplitz operators \cite[Theorem 1]{BG81}.
\begin{lem}\cite{EnE15}\label{lpq} Let $T_Q$ be a generalized Toeplitz operator of negative order $-q$ on $\partial \Omega,$ then $T_Q\in \mathcal{L}^{p}$ all $p>{\text{dim}_{\mathbb{C}}\hspace{0.1em}\Omega}/{q}.$
\end{lem}

We now state the following result about the restriction operators and the Szeg\"{o} projection.
\begin{lem}\cite{EnE15}\label{eTr} For each $i=1,\cdots,m,$ $R_{\partial V_i}\varPi=\varPi_{V_i}R_{\partial V_i} \varPi$ is an elliptic Fourier integral operator from $\partial\Omega$ to $ \partial\Omega\cap V_i$ of order $k_i/{2}.$ Moreover, if $T$ is a generalized Toeplitz operator on $\partial\Omega$ of order $s \in\mathbb{R}$, then $(R_{\partial V_i}\varPi)T (R_{\partial V_i}\varPi)^\ast$ is a generalized Toeplitz operator on $\partial\Omega\cap V_i$ of order $s + k_i$, which is elliptic if $T$ is.
\end{lem}

It follows from  (\ref{norsm}) and (\ref{norms}) that there exists a  positive elliptic generalized Toeplitz operator $T_{Y}$  of order  $2s -1$ on $\partial\Omega$ such that the inner product in $ \mathcal{O}^{s}(\overline{\Omega})$  is induced by  $$\langle f,g\rangle_ {\mathcal{O}^{s}(\overline{\Omega})}=\langle T_{Y}\gamma f,\gamma g\rangle_ {\mathcal{O}^{0}(\partial\Omega)}.$$  Similarly,  there exists a positive elliptic generalized Toeplitz operator of order $2s-k_i-1$ on $ \partial\Omega\cap V_i$ such that the inner product in $ \mathcal{O}^{s-k_i/2}(\overline{\Omega}\cap V_i)$ is induced by  \begin{equation}\label{Tiin}\langle f,g\rangle_ {\mathcal{O}^{s-k_i/2}(\overline{\Omega}\cap V_i)}=\langle T_{X_i}\gamma_{V_i}f,\gamma_{V_i}g\rangle_ {\mathcal{O}^{0}(\partial\Omega\cap V_i)}.\end{equation}   
Denote by $T=(T_1,\cdots,T_m)^{\prime}$ the column block matrix of operators with \begin{equation}\label{Tij}T_i=T_{X_i}^{\frac{1}{2}}\gamma_{V_i}R_{V_i}=T_{X_i}^{\frac{1}{2}}R_{\partial V_i}\gamma:  \mathcal{O}^{s}(\overline{\Omega})\rightarrow \mathcal{O}^{0}(\partial\Omega\cap {V_i}),\end{equation} $i=1,\cdots,m.$ Clearly,  $T$ maps $\mathcal{O}^{s}(\overline{\Omega})$ into $\mathcal{O}^{0}_V:=\bigoplus_{i=1}^m \mathcal{O}^{0}(\partial\Omega\cap {V_i}),$ where the element in $\bigoplus_{i=1}^m \mathcal{O}^{0}(\partial\Omega\cap V_i)$ is viewed as the column block matrix of functions.

\begin{prop}\label{tfij} Suppose as above.\begin{enumerate}
\item  $T_i$ is right semi-Fredholm for $i=1,\cdots,m.$
\item  There exists an elliptic generalized Toeplitz operator $T_{Q_i}$ on $\partial\Omega\cap V_i$ of order $-2s+k_i+1$ such that
$$T_iT_i^\ast=T_{X_i}^{\frac{1}{2}}T_{Q_i}T_{X_i}^{\frac{1}{2}}$$ for $i=1,\cdots,m.$
\item If  $V_i\cap V_j \subset \Omega$  whenever $1\leq i\neq j\leq m, $ then $T_iT_j^\ast$  is a smoothing operator from $\mathcal{O}^{0}(\partial\Omega\cap V_j)$ to $\mathcal{O}^{0}(\partial\Omega\cap V_i).$
\end{enumerate}
\begin{proof}  (1) $T_{X_i}^{\frac{1}{2}}$ is Fredholm by (P6), $R_{V_i}$ is right semi-Fredholm, and $\gamma_{V_i}$ is invertible. The desired result is immediate from the following fact: Let $H_i,i=1,\cdots,4,$ be Hilbert spaces and  $B_i: H_i\rightarrow H_{i+1},i=1,2,3,$ are bounded linear operators.  If   ~$B_1, B_3$ are right semi-Fredholm and $B_2$ has bounded inverse, then the operator $B_3B_2B_1:H_1\rightarrow H_4$  is right semi-Fredholm.
 
Indeed, by \cite[Thorem XI.2.3]{Coy90} and its remark, it follows that there exist bounded operators $B_i^\prime: H_{i+1}\rightarrow H_{i}$ and finite rank operators $F_i: H_{i+1}\rightarrow H_{i+1},i=1,3,$ such that $$B_iB_i^\prime=\text{Id}_{H_{i+1}}+F_i,\quad i=1,3.$$
Thus $$(B_3B_2B_1)(B_1^\prime B_2^{-1}B_3^\prime)=\text{Id}_{H_4}+F_3+B_3B_2F_1B_2^{-1}B_3^\prime.$$
By \cite[Thorem XI 2.3]{Coy90} again, this implies that $B_3B_2B_1$ is right semi-Fredholm. 

(2) For any $u\in  \mathcal{O}^{0}(\partial\Omega\cap V_i)$ and $v\in\mathcal{O}^{s}(\overline{\Omega}),$ \begin{equation}\begin{split}\label{}
<T_i^\ast u,v>_ {\mathcal{O}^{s}(\overline{\Omega})}&=< u,T_iv>_ { \mathcal{O}^{0}(\partial\Omega\cap V_i)}\\ \notag
&=<u,T_{X_i}^{\frac{1}{2}}\gamma_{V_i}R_{V_i}v>_ { \mathcal{O}^{0}(\partial\Omega\cap V_i)}\\
&=<u,T_{X_i}^{\frac{1}{2}}R_{\partial V_i}\gamma v>_ { \mathcal{O}^{0}(\partial\Omega\cap V_i)}\\
&=<u,T_{X_i}^{\frac{1}{2}}R_{\partial V_i}\varPi\gamma v>_ { \mathcal{O}^{0}(\partial\Omega\cap V_i)}\\
&=<(R_{\partial V_i}\varPi)^\ast T_{X_i}^{\frac{1}{2}}u,\gamma v>_ { \mathcal{O}^{0}(\partial\Omega)}.\\
 \end{split}
 \end{equation} 
On the other hand, $$<T_i^\ast u,v>_ {\mathcal{O}^{s}(\overline{\Omega})}=<T_{Y}\gamma T_i^\ast u,\gamma v>_ { \mathcal{O}^{0}(\partial\Omega)}.$$ Hence, $T_i^\ast =KT_{Y_i}^{-1}(R_{\partial V_i}\varPi)^\ast T_{X_i}^{{1}/{2}}$ and $$ T_i T_i^\ast=T_{X_i}^{\frac{1}{2}}(R_{\partial V_i}\varPi )T_{Y}^{-1}(R_{\partial V_i}\varPi)^\ast T_{X_i}^{{1}/{2}},i=1,\cdots,m.$$
Put $T_{Q_i}:= (R_{\partial V_i}\varPi) T_{Y_i}^{-1}(R_{\partial V_i}\varPi)^\ast,i=1,\cdots,m .$  It follows from Lemma \ref{eTr} that $T_{Q_i}$ is an elliptic generalized Toeplitz operator of order $-2s+k_i+1$ for each $1\leq i\leq m,$ which leads to the desired.

(3) It follows from \cite[Theorem 1.5]{BS81} that the Szeg\"o kernel $S(x,y)$ of the compact strongly pseudoconvex boundary $\partial{\Omega}$ is a Fourier integral distribution of order zero, namely  there exists an  elliptic symbol $a\in S^{d-1}(\overline{\Omega}\times\overline{\Omega}\times\mathbb{R})$ so that $$S(x,y)=\int_0^\infty e^{\sqrt{-1} \rho(x,y) t}a(x,y,t)dt,$$
where $\rho(x,y)$ is called  the almost analytic extension of the defining function $\rho$ and is determined by the following properties: $\rho(x,x)=\frac{1}{\sqrt{-1}}\rho(x),\rho(x,y)=-\overline{\rho(y,x)},$  and  $\bar{\partial}_x\rho(x,y), {\partial}_y\rho(x,y)$ vanish to infinite order along the diagonal. Thus the Szeg\"o projection $\varPi: L^2(\partial{\Omega}) \rightarrow \mathcal{O}^0(\partial{\Omega})$ is an elliptic Fourier integral operator of order zero with complex-valued phase function.  Note from  (\ref{reci}) that the operator $R_{\partial V_i}$ is the restriction of the operator $R_{V_i}$ to the boundary $\partial{\Omega}\cap V_i$, $i=1,\cdots,m.$ Combining the fact that the restriction operator $R_{V_i}$ maps holomorphic functions on $\overline{\Omega}$ to holomorphic functions on the submanifold $V_i,$ it follows that the Schwartz kernel of $R_{\partial V_i}\varPi$ is the restriction with respect to the variable $x,$
$$R_{\partial V_i} S(x,y)=\int_0^\infty e^{\sqrt{-1} \rho(x,y) t}a(x,y,t)dt,\quad x \in\partial{\Omega}\cap V_i, y\in  \partial{\Omega}.$$
Thus $$(R_{\partial V_i}\varPi)f(x)=\int_{\partial{\Omega}} S(x,y)f(y)i(w^d),$$ and 
$$(R_{\partial V_i}\varPi)^\ast l(y)=\int_{\partial{\Omega}\cap V_i} S(y,x)l(x)i(w^d),$$ where $i(w^d)$ is the volume form on $\partial{\Omega}$ naturally induced from $(W, w^d).$
Combined with Fubini's theorem, it follows that 
$$T_Y^{-1} (R_{\partial V_j}\varPi)^\ast l(y)= \int_{\partial{\Omega}\cap V_j} T_Y^{-1} S(y,x)l(x)i(w^d). $$
Then $$ (R_{\partial V_i}\varPi)T_Y^{-1} (R_{\partial V_j}\varPi)^\ast l(y)= \int_{x\in \partial{\Omega}\cap V_j} l(x)\int_{u\in\partial{\Omega}} S(y,u) T_Y^{-1} S(u,x)i(w^d)i(w^d), $$ where $T_Y^{-1} $ acts on the $u$ variable. 
By the reproducing property of the Szeg\"o kernel, it follows that the  Schwartz kernel of $ (R_{\partial V_i}\varPi)T_Y^{-1} (R_{\partial V_j}\varPi)^\ast$ is  $T_{Y}^{-1}S(y,x),$  where $T_{Y}^{-1}$ acts on the $y$ variable. Hence, the Szeg\"o kernel of $$ T_iT_j^\ast = T_{X_i}^{\frac{1}{2}}(R_{\partial V_i}\varPi )T_{Y}^{-1}(R_{\partial V_i}\varPi)^\ast T_{X_i}^{{1}/{2}} $$ is \begin{equation}\label{swzk} T_{X_i}^{\frac{1}{2}} (T_{X_j}^{\frac{1}{2}}\otimes \text{Id}_{\overline{\mathcal{O}^0(\partial{\Omega}\cap V_j)}})T_{Y}^{-1}\hspace{0.1em} S(y,x) \end{equation}  
  for $y\in \partial{\Omega}\cap V_j$ and $ x\in \partial{\Omega}\cap V_i,$ where $T_{X_i}^{\frac{1}{2}}, T_{Y}^{-1}$  act on the $u$ variable, $T_{X_j}^{\frac{1}{2}}, \text{Id}_{\overline{\mathcal{O}^0(\partial{\Omega}\cap V_j)}}$ act on the $x$ variable, and the tensor product  $T_{X_j}^{\frac{1}{2}}\otimes\text{Id}_{\overline{\mathcal{O}^0(\partial{\Omega}\cap V_j)}}$ act on the Hilbert space $\mathcal{O}^0(\partial{\Omega}\cap V_j) \otimes\overline{\mathcal{O}^0(\partial{\Omega}\cap V_j)}$ in the usual sence, $\overline{\mathcal{O}^0(\partial{\Omega}\cap V_j)}$ is the Hilbert space consisting of conjugate functions of $\mathcal{O}^0(\partial{\Omega}\cap V_j).$
Note that the Szeg\"o kernel $S(y,x)$ is smooth on $(\partial{\Omega}\cap V_j)\times(\partial{\Omega}\cap V_i)$ whenever $i\neq j,$ since 
 $(\partial{\Omega}\cap V_j)\cap(\partial{\Omega}\cap V_i)=\emptyset.$
 It then implies from (\ref{swzk}) that the Schwartz kernel of $T_iT_j^\ast$ is smooth whenever $i\neq j.$  This completes the proof.
\end{proof}

\begin{cor}\label{frlm} Let $D_T$ be the diagonal operator matrix $\text{diag} \hspace{0.1em} (T_1 T_1^\ast,\cdots,T_mT_m^\ast)$ on $\bigoplus_{i=1}^m \mathcal{O}^{0}(\partial\Omega\cap V_i),$ then $D_T$ is a Fredholm operator.
\end{cor}
\end{prop}

 \section{The Proof of Theorem \ref{main}}
  This section is devoted to the proof of Theorem \ref{main}. We now explain briefly the main ideas of the proofs. Our proof is inspired by the method due to Engli\v{s} and Eschmeier \cite{EnE15} to handle the case where a homogeneous algebraic variety with only one isolate singularity can be resolved by only one blowing-up.  But, in general, the singularity of an analytic subvariety is not isolated, or it can not be resolved by only one blowing-up even in the algebraic case with only one isolated singularity. However, Hironaka's embedded desingularization theorem \cite{AHV18, Hir64} shows that the singularities of an analytic subvariety can always be resolved by finitely many permissible blowing-ups locally along suitable smooth centers contained in the singularities. Note that the strong pseudoconvexity of the boundary preserves under the permissible blowing-up in the case of compact singularity, it reduces to the smooth case by proving that the  $p$-essential normality is preserved invariant under permissible blowing-ups.  We first deal with the case where $V$ is smooth. 
 
 \begin{prop}\label{pnor}  Suppose that $V$ is smooth and has no compact irreducible components, then the   Bergman-Sobolev quotient submodule  $M_V^\bot\subset \mathcal{O}^{s}(\overline{\Omega})$ is  $p$-essentially normal for all $p>\text{dim}_{\mathbb{C}}\hspace{0.1em} (V\cap\Omega).$
 \end{prop}
 \begin{proof} Keeping notations in Section \ref{ser}, we first prove that $$M_V=\text{Ker}\hspace{0.1em} T= \bigcap_{i=1}^m \text{Ker}\hspace{0.1em} T_i.$$ It suffices to show that $\text{Ker}\hspace{0.1em} T_i=\text{Ker}\hspace{0.1em} R_{V_i}=\{f\in  \mathcal{O}^{s}(\overline{\Omega}): f|_{V_i\cap\Omega}=\bm{0}\}$ by Remark \ref{KRV}, $i=1,\cdots,m.$ Indeed, this is immediate from (\ref{Tij}) and the equivalent inner product (\ref{Tiin}) in $\mathcal{O}^{s-k_i/2}(\partial\Omega\cap V_i).$ On the other hand, the $m\times m$ operator matrix $$TT^\ast: \mathcal{O}^{0}_V\rightarrow \mathcal{O}^{0}_V $$ can be written as
 \begin{equation}\begin{split}\label{mtx} TT^\ast&=\begin{bmatrix}
 T_1 T_1^\ast & T_1T_2^\ast&\cdots&T_1T_m^\ast\\
 T_2 T_1^\ast & T_2T_2^\ast&\cdots&T_2T_m^\ast\\
 \vdots&\vdots&\ddots&\vdots\\
 T_m T_1^\ast & T_mT_2^\ast&\cdots&T_mT_m^\ast\\
 \end{bmatrix}\\
 &\overset{\text{def}}{=} D_T+C,
 \end{split}
 \end{equation}
 where $D_T=\text{diag} \hspace{0.1em} (T_1 T_1^\ast,\cdots,T_mT_m^\ast)$ is the diagonal operator matrix and $C=TT^\ast-D_T.$
Note that $T_iT_j^\ast$  is smoothing whenever $i\neq j$ from Proposition \ref{tfij} (3),  combining with (P4) and Lemma \ref{lpq},  this implies that $C$ is compact.
Since $D_T$ is Fredholm by Corollary \ref{frlm}, it follows from (\ref{mtx}) that $TT^\ast$ is also Fredholm. Then there exists a continuous operator $B$ and a finite rank operator $F$ on $\mathcal{O}^{0}_V$ satisfying $ BTT^\ast =\text{Id}+F$ where $\text{Id}$ is the identity operator on $\mathcal{O}^{0}_V,$  thus $T^\ast$ is left semi-Fredholm by \cite[Theorem 2.3]{Coy90},
which means that $T$ is right semi-Fredholm as we mentioned before. It follows that its range $\text{Ran} \hspace{0.1em} T$  is closed in $\mathcal{O}^{0}_V$ and has finite codimension.  Thus the restriction $\tau$ of $T$ to $M_V^\bot=(\text{Ker}\hspace{0.1em} T)^\bot$ maps onto  the closed subspace $\text{Ran} \hspace{0.1em} T\subset \mathcal{O}^{0}_V.$ The Banach inverse mapping  theorem implies that $\tau:M_V^\bot\rightarrow\text{Ran} \hspace{0.1em} T$ is invertible. Denote $G$ by the diagonal operator matrix $ \text{diag} \hspace{0.1em} (\tau\tau^\ast,I_{(\text{Ran}\hspace{0.1em} T)^\bot}) $ which acts on $\mathcal{O}^{0}_V=\text{Ran}\hspace{0.1em} T\oplus (\text{Ran}\hspace{0.1em} T)^\bot=\text{Ran}\hspace{0.1em} T\times (\text{Ran}\hspace{0.1em} T)^\bot,$ where $I_{(\text{Ran}\hspace{0.1em} T)^\bot}$ is the identity operator on the finite dimensional space $(\text{Ran}\hspace{0.1em} T)^\bot.$ Clearly, $G$ is continuous and invertible. We now claim that \begin{equation}\label{pgt}\langle P_{M_V^\bot}u,v\rangle_{\mathcal{O}^s(\overline{\Omega})}=\langle G^{-1}Tu,Tv\rangle_{\mathcal{O}^{0}(\partial\Omega)}\end{equation} for all $u,v\in \mathcal{O}^s(\overline{\Omega}).$ Indeed, both sides of (\ref{pgt}) vanish if $f$ or $g$ belong to $M_V=\text{Ker} \hspace{0.1em} T.$ Otherwise, when both $u$ and $v$ belong to $M_V^\bot=(\text{Ker} \hspace{0.1em} T)^\bot,$ thus $\tau u=Tu,\tau v=Tv\in \text{Ran}\hspace{0.1em} T,$ by the definition of $G,$ we have $$G^{-1}Tu=(\tau^\ast)^{-1}\tau^{-1}\tau u=(\tau^\ast)^{-1}u.$$
Hence, $$\langle G^{-1}Tu,Tv\rangle_{\mathcal{O}^{0}(\partial\Omega)}=\langle (\tau^\ast)^{-1}u,\tau v\rangle_{\mathcal{O}^{0}(\partial\Omega)}=\langle u,v\rangle_{\mathcal{O}^s(\overline{\Omega})}=\langle P_{M_V^\bot}u,v\rangle_{\mathcal{O}^s(\overline{\Omega})}.$$ This proves (\ref{pgt}).
Thus, for arbitrary $f\in \mathcal{O}(\overline{\Omega}),$ the compression   $$S_f= P_{M_V^\bot} M_f=T^\ast G^{-1}TM_f.$$  
Since $D_T$ is a diagonal operator matrix of elliptic generalized Toeplitz operators,   (P6) implies that there exists a diagonal operator matrix $ \tilde{T}:= \text{diag} \hspace{0.1em} (\tilde{T}_1,\cdots,\tilde{T}_m)$ of elliptic generalized Toeplitz operators $\tilde{T}_i$ with $\text{ord} \hspace{0.1em} \tilde{T}_i=-\text{ord} \hspace{0.1em} T_iT_i^\ast=0,i=1\cdots,m,$ such that  $\tilde{C}:=D_T\tilde{T}-I$ is a diagonal operator matrix of smoothing operators.
Combining (\ref{mtx}) with the definition of $G$ implies that $$G\tilde{T}=I+ \text{diag} \hspace{0.1em} (\bm{0},I_{(\text{Ran}\hspace{0.1em} T)^\bot}) \cdot\tilde{T}+C\tilde{T}+\tilde{C}.$$ Since $C,\text{diag} \hspace{0.1em} (\bm{0}, I_{(\text{Ran}\hspace{0.1em} T)^\bot})$  are finite ranks operator and  $\tilde{C}$ is a diagonal operator matrix of smoothing operators, it follows from Lemma \ref{lpq} that  $\tilde{T}$ can be written as the form of $\tilde{T}=G^{-1}+F_1,$ where $F_1 \in \mathcal{L}^{p}(\mathcal{O}^{0}_V)$ for all $p>0.$
Note that $$T_iM_f=T_{X_i}^{\frac{1}{2}}\gamma_{V_i}R_{V_i}M_f=T_{X_i}^{\frac{1}{2}}M_{f|_{\partial\Omega\cap V_i}}\gamma_{V_i}R_{V_i}=T_{X_i}^{\frac{1}{2}}M_{f|_{\partial\Omega\cap V_i}}T_{X_i}^{-\frac{1}{2}}T_i,$$ where  $i=1,\cdots,m.$  Since $f\in  \mathcal{O}(\overline{\Omega}),$ it  immediate that $M_{f|_{\partial\Omega\cap V_i}}$ identifies the classical Toeplitz operator $T_{f|_{\partial\Omega\cap V_i}}$ with the function symbol $f|_{\partial\Omega\cap V_i},$ note that the multiplication operator  $M_{f|_{\partial\Omega\cap V_i}}$ is a pseudodifferential operator of order $0.$ Thus $M_{f|_{\partial\Omega\cap V_i}}=T_{f|_{\partial\Omega\cap V_i}}$ is actually a generalized Toeplitz operator, whose order $\text{ord} \hspace{0.1em} T_{f|_{\partial\Omega\cap V_i}}\leq0$ by (P2). Then we have
\begin{equation}\begin{split}\label{}
TM_f&=(T_1M_f,\cdots,T_mM_f)^{\prime}\\ \notag
&=(T_{X_1}^{\frac{1}{2}}M_fT_{X_1}^{-\frac{1}{2}}T_1,\cdots,T_{X_m}^{\frac{1}{2}}M_fT_{X_m}^{-\frac{1}{2}}T_m)^{\prime}\\
&=(T_{X_1}^{\frac{1}{2}}T_{f|_{\partial\Omega\cap V_i}}T_{X_1}^{-\frac{1}{2}}T_1,\cdots,T_{X_m}^{\frac{1}{2}}T_{f|_{\partial\Omega\cap V_i}}T_{X_m}^{-\frac{1}{2}}T_m)^{\prime}\\
\end{split}
 \end{equation}
 is a $m\times 1$  operator matrix. Therefore, we get $$S_f=T^\ast G^{-1}TM_f=\tau^\ast  \cdot\text{diag} \hspace{0.1em} (\tilde{T}_1T_{X_1}^{\frac{1}{2}}T_{f|_{\partial\Omega\cap V_i}}T_{X_1}^{-\frac{1}{2}},\cdots,\tilde{T}_mT_{X_m}^{\frac{1}{2}}T_{f|_{\partial\Omega\cap V_i}}T_{X_m}^{-\frac{1}{2}}) \cdot\tau+C_f,$$ where $C_f \in \mathcal{L}^{p}(M_V^\bot)$ for all $p>0.$
For brevity, we denote $$\mathcal{T}_f=\text{diag} \hspace{0.1em} (\tilde{T}_1T_{X_1}^{\frac{1}{2}}T_{f|_{\partial\Omega\cap V_i}}T_{X_1}^{-\frac{1}{2}},\cdots,\tilde{T}_mT_{X_m}^{\frac{1}{2}}T_{f|_{\partial\Omega\cap V_i}}T_{X_m}^{-\frac{1}{2}}),$$  which is a diagonal operator matrix of generalized Toeplitz operators.  We now estimate the order of its entries on the main diagonal. By the addition  formula of the order in (P3), it follows that $$\text{ord} \hspace{0.1em} (\tilde{T}_iT_{X_i}^{\frac{1}{2}}T_{f|_{\partial\Omega\cap V_i}}T_{X_i}^{-\frac{1}{2}} )=\text{ord} \hspace{0.1em} \tilde{T}_i+\text{ord} \hspace{0.1em}T_{X_i}^{\frac{1}{2}}+\text{ord} \hspace{0.1em}T_{f|_{\partial\Omega\cap V_i}}+\text{ord} \hspace{0.1em}T_{X_i}^{-\frac{1}{2}} \leq 0,$$  $i=1,\cdots,m.$

For any $f,g\in \mathcal{O}(\overline{\Omega}),$ it follows from the following Lemma \ref{age} that
\begin{equation}\begin{split}\label{Sfg}
[S_f,S_g^\ast]&=[\tau^\ast\mathcal{T}_f\tau+C_f,\tau^\ast\mathcal{T}_g^\ast\tau+C_g^\ast]\\
&=\tau^\ast (\mathcal{T}_f[ D_T,\mathcal{T}_g^\ast]+[\mathcal{T}_f,\mathcal{T}_g^\ast D_T])\tau+C_{f,g}\\
&=\tau^\ast (\mathcal{T}_f[ D_T,\mathcal{T}_g^\ast]+[\mathcal{T}_f,\mathcal{T}_g^\ast]D_T+\mathcal{T}_g^\ast[\mathcal{T}_f,D_T])\tau+C_{f,g},\\
\end{split}
 \end{equation}
 where $C_{f,g} \in \mathcal{L}^{p}(M_V^\bot)$ for all $p>0.$ Since the last three commutators in (\ref{Sfg}) are diagonal operator matrixes with generalized Toeplitz operators on the main diagonal and zeroes elsewhere; combining with (P4) and (P5), it follows that the order of generalized Toeplitz operators on the main diagonal are all no more than $-1.$
 It follows from Lemma \ref{lpq} that $[S_f,S_g^\ast]\in \bigoplus_{i=1}^m  \mathcal{L}^{p_i}(\mathcal{O}^0(\partial\Omega\cap V_i))$  for all $p_j>\text{dim}_{\mathbb{C}} \hspace{0.1em} (V_i\cap \Omega),i=1,\cdots,m.$ Thus $[S_f,S_g^\ast]\in  \mathcal{L}^{p}(M_V^\bot )$ for all $$p>\max\{\text{dim}_{\mathbb{C}} \hspace{0.1em} (V_1\cap\Omega),\cdots,\text{dim}_{\mathbb{C}} \hspace{0.1em} (V_m\cap\Omega)\}=\text{dim}_{\mathbb{C}} \hspace{0.1em} (V\cap\Omega).$$ The proof is complete.
 \end{proof}

 The following identity related to the commutator is elementary, which can be immediately verified by direct calculations.
\begin{lem}\label{age} Let $\mathcal{R}$ be a ring and $U_1,U_2,U_3\in \mathcal{R},$  then the following   identity holds:
 $$ [U_1U_2,U_3]=U_1[U_2,U_3]+[U_1,U_3]U_2.$$
\end{lem}

We now turn to the more general case where $V\cap \Omega \subset \Omega$ with compact singularities. According to Hironaka's embedded desingularization theorem \cite{AHV18, Hir64}, there exists a resolution map from a complex manifold $\hat{\Omega}$ to $\Omega$ such that the strict transform of $V\cap \Omega$  is smooth and satisfies additional conditions. The above resolution is called an embedded resolution, which is composed of a finite sequence of permissible blowing-up.  Thus we will first investigate the property of $p$-essential normality of  Bergman-Sobolev quotient modules under permissible blowing-up.

Now let $\eta:\check{\Omega}\rightarrow{\Omega}$ be a blowing-up along with the smooth center $E$ is called permissible if $E$ is contained in the singular locus $\text{Sing}(V).$ Note that the blowing-up map is proper and $\text{Sing}(V)$ is compact as we assumed, this implies that $\eta^{-1}(E),\eta^{-1}(\text{Sing}(V))$ and  the closure $\overline{\check{\Omega}}$ are both  compact.  Since $\text{Sing}(V)\cap \partial\Omega=\emptyset$ and $\eta: \check{\Omega}\setminus \eta^{-1}(E)\rightarrow \Omega\setminus E$ is biholomorphic, it follows that the boundary $\partial\check{\Omega}$ is also smooth strongly pseudoconvex by the biholomorphic invariance of the strong pseudoconvexity. It follows from the definition of the strict transform $ V^{\prime}\subset \check{\Omega}$ of $V$ that $\text{Sing}(V^\prime)\subset\eta^{-1}(\text{Sing}(V)),$ thus the strict transform $ V^{\prime}$ has at most compact singularities and the non-compact part of the regular part $ \text{Reg}(V^{\prime})$ transversely intersects the smooth strongly pseudoconvex boundary $\partial{\check{\Omega}}$. Clearly, the pull back $\eta^\ast(\rho)=\rho\circ\eta$ is a defining function of $\check{\Omega},$  thus $\eta^\ast(\nu)=\eta^\ast(\rho) \land(\textbf{d} \eta^\ast(\rho))^{d-1}$ is a smooth positive $1$-density on $\check{\Omega}.$ Though the pullback $\eta^\ast(\omega)$ is may not positive on $\eta^{-1}(E),$  the pullback of volume form $\eta^\ast({\omega^d})$ is still a volume form (viewed as a measure) on $\check{\Omega}.$ 
 Now we take an arbitrary Hermitian metric $\hat{w}$ on $\check{\Omega}$ and a smooth function $\theta$ with compact support in $\check{\Omega}$ such that $\theta$ is non-vanishing on the compact set $\eta^{-1}(E).$  Denote by $\check{w}= \eta^\ast(w)+\theta \hat{w}.$ Clearly, $\check{w}$ is a Hermitian metric  on $\check{\Omega}$ and
\begin{equation}\label{whhw}\check{w}^d= (\eta^\ast(w)+\theta \hat{w})^d=\eta^\ast(w)^d+\sum_{k=1}^{d} \theta^k\binom{d}{k}\eta^\ast(w)^{d-k}\wedge\hat{w}^k,\notag\\ \end{equation} where $\binom{d}{k}=\frac{d!}{k!(d-k)!}$ is the binomial coefficient. Note that the volume forms $ \theta^k\eta^\ast(w)^{d-k}\wedge\hat{w}^k$ have compact supports in $\check{\Omega},$ it follows that the induced volume form $i(\eta^\ast(w)^d)=i(\check{w}^d)$ on $\partial{\check{\Omega}},$ and hence $L^2(\partial{\check{\Omega}},i(\eta^\ast(w)^d))=L^2(\partial{\check{\Omega}},i(\check{w}^d)).$ Let $$\check{K}:L^2(\partial{\check{\Omega}},\eta^\ast(\nu)){\longrightarrow}  L^2({\check{\Omega}},\check{w}^d)$$  the related Poisson operator on the Hermitian manifold $(\check{\Omega},\check{w}).$  Since $M_{\vert \eta^\ast(\rho) \vert^{2s} } \check{K}$ is Poisson operator of order $-2s$ and $\check{K}^\ast M_{\vert \eta^\ast(\rho) \vert^{2s} } $ is a trace operator $-2s-1$ for each $s>-{1}/{2}$ by  \cite[page 257]{Bou79}, it follows that the operator  $$\check{\Lambda}_{s}:=\check{K}^\ast M_{\vert \eta^\ast(\rho) \vert^{s} } \check{K}$$  is a positive elliptic pseudodifferential operator of order $-s-1$ for  $s>-1.$ We define $\check{B}_s=F_s(\check{\Lambda}_0^{-1})$  by continuous function calculus for (unbounded) selfadjoint operators for $s\leq-1.$ Then \cite[Theorem XII.1.3]{Tay81} for functional calculus of pseudodifferential operators implies  that $\check{B}_s$ is a positive elliptic pseudodifferential operator of order $-s-1$ for every $s\leq-1.$ Similar to (\ref{norsm}), we define a family $\{\Vert\cdot\Vert_{\partial\check{\Omega},s} \}_{s\in \mathbb{R}}$ of norms on $\mathcal{O}(\partial\check{\Omega})$ as follows:
 
 \begin{equation}\label{norssm} \Vert u\Vert_{\partial\check{\Omega},s}^2= \begin{cases}
{\langle \check{\Lambda}_{-2s-1} u,u\rangle_{\partial\check{\Omega},i(\check{w}^d)}}, & s<0 ;\\   
 {\langle  u,u\rangle_{\partial\check{\Omega},i(\check{w}^d)}},&  s=0;\\
  {\langle \check{B}_{-2s-1} u,u\rangle_{\partial\check{\Omega},i(\check{w}^d)}},& s>0.\\
\end{cases}
\end{equation}

From the above, we see that $\Vert\cdot\Vert_{\partial\check{\Omega},s}$ is an  equivalent norm of $\mathcal{O}^s(\partial\check{\Omega})$ for every $s\in\mathbb{R}.$ Denote the norm $\Vert\cdot\Vert_{\check{\Omega},s}$ on $\mathcal{O}^{s}(\overline{\check{\Omega}})$ as \begin{equation}\label{normss}\Vert u\Vert_{\check{\Omega},s}=\Vert\gamma_{\check{\Omega}} u\Vert_{\partial\check{\Omega},s-1/2}\end{equation} for every $s\in\mathbb{R}.$ It follows from (\ref{krps}) that $\Vert\cdot\Vert_{\check{\Omega},s}$ is an  equivalent norm of $\mathcal{O}^{s}(\overline{\check{\Omega}})$  for every $s\in\mathbb{R}.$ The norms defined in (\ref{norssm}) and (\ref{normss}) are all called canonical norms (concerning the blowing-up $\eta$). 

\begin{lem}\label{luni} For $s>-1,$ $ \Lambda_s:L^2(\partial \Omega,i(w^d))\rightarrow L^2(\partial \Omega,i(w^d))$ and  $ \check{\Lambda}_s:L^2(\partial \check{\Omega},i(\check{w}^d))\rightarrow L^2(\partial \check{\Omega},i(\check{w}^d))$ are unitary, namely $\check{\Lambda}_s=\eta^\ast\Lambda_s(\eta^\ast)^{-1}.$
\end{lem}
\begin{proof} Since $\eta^\ast: L^2(\partial \Omega,i(w^d))\rightarrow L^2(\partial \check{\Omega},i(\check{w}^d))$ is unitary, it suffices to prove that $$\langle (\eta^\ast)^{-1}\check{\Lambda}_s \eta^\ast u,v\rangle_{\partial \Omega,i(w^d)} = \langle \Lambda_su,v\rangle_{\partial \Omega,i(w^d)}$$ for all $u,v\in L^2(\partial \Omega,i(w^d)).$
It can be directly checked that  $\check{K}\eta^\ast= \eta^\ast K$ on $L^2(\partial \Omega,i(w^d))$ and $\int_{\eta^{-1}(E)}\eta^\ast(w^d)=0.$
And hence, \begin{equation}\begin{split}\label{}
\langle (\eta^\ast)^{-1}\check{\Lambda}_s \eta^\ast u,v\rangle_{\partial \Omega,i(w^d)}&= \langle \check{\Lambda}_s \eta^\ast u, \eta^\ast v\rangle_{\partial \check{\Omega},i(\check{w}^d)}\\ \notag
&=\langle  M_{\vert \eta^\ast(\rho) \vert^{s} } \check{K} \eta^\ast u, \check{K} \eta^\ast v\rangle_{\check{\Omega},\eta^\ast(w)^d}\\
&=\langle  M_{\vert \eta^\ast(\rho) \vert^{s} }\eta^\ast {K}  u, \eta^\ast {K} v\rangle_{\check{\Omega},\eta^\ast(w)^d}\\
&=\langle  M_{\vert \rho\vert^{s}} {K}  u,  {K} v\rangle_{\Omega,w^d}\\
&=\langle  \Lambda_s   u,   v\rangle_{\partial \Omega,i(w^d)}.\\
\end{split}\end{equation} 
This completes the proof.
\end{proof}

\begin{prop}\label{blh} 
(1)  A holomorphic function $f\in\mathcal{O}(\Omega)$ satisfies  $f|_{V\cap\Omega}=\bm{0}$ if and only if $\eta^\ast(f)|_{V^\prime\cap \check{\Omega}}=\bm{0}.$

(2) The  permissible blowing-up $\eta:\check{\Omega}\rightarrow{\Omega}$  introduces a unitary operator $$ \eta^\ast:\mathcal{O}^{s}(\overline{\Omega}) \rightarrow  \mathcal{O}^{s}(\overline{\check{\Omega}}), \quad f\longmapsto \eta^\ast(f)=f\circ \eta,$$ with canonical norms for all $s\in\mathbb{R}.$
\end{prop}
\begin{proof} (1) By definition, it is clear that $f|_{V\cap \Omega}=\bm{0}$ if  $\eta^\ast(f)|_{V^\prime\cap\check{\Omega}}=\bm{0}.$ We now prove the converse.
 It comes from the following decomposition $$V^\prime\cap \check{\Omega}= \big((V^\prime\cap \check{\Omega})\cap \eta^{-1}(E)\big)\cup \big((V^\prime\cap \check{\Omega})\setminus\eta^{-1}(E)\big),$$
and $\eta: (V^\prime\cap \check{\Omega})\setminus \eta^{-1}(E)\rightarrow (V\cap \Omega)\setminus E$ is biholomorphic that  $\eta^\ast(f)|_{(V^\prime\cap \check{\Omega})\setminus \eta^{-1}(E)}=\bm{0}$ if and only if $f|_{(V \cap\Omega)\setminus E}=\bm{0}.$  Suppose that there exists a point $x^\prime\in (V^\prime \cap \check{\Omega})\cap \eta^{-1}(E)$ satisfying $\eta^\ast(f)(x^\prime)\neq0,$ namely, $f(\eta(x^\prime))\neq 0.$ Since $\eta(x^\prime)\in E\subset \text{Sing}(V\cap {\Omega})$ and $\text{Sing}(V\cap {\Omega})$ is nowhere dense in $V\cap {\Omega}$, it follows that there exists a point $ x_0\in \text{Reg}(V\cap {\Omega})\subset V\cap {\Omega}\setminus E$ such that $f(x_0)\neq0,$ which contradicts $f|_{(V\cap {\Omega})\setminus E}=\bm{0}.$

(2)
 We first prove that $ \eta^\ast: \mathcal{O}(\Omega)\rightarrow\mathcal{O}(\check{\Omega})$ is bijective.  Suppose that there exist $f,g \in \mathcal{O}(\Omega)$ satisfying $\eta^\ast(f)=\eta^\ast(g)$ on $\check{\Omega}.$ Since $E,\eta^{-1}(E)$ are nowhere sense, it follows that there exists a nonempty open set $U\subset \check{\Omega}\setminus \eta^{-1}(E)$ satisfying $\eta^\ast(f)=\eta^\ast(g)$ on $U.$
Observe that  $\eta: \check{\Omega}\setminus \eta^{-1}(E)\rightarrow \Omega\setminus E$ is biholomorphic, this implies that $f=g$ on the nonempty open subset $\eta(U),$ thus $f=g $ on whole $\Omega$ by the uniqueness theorem of holomorphic functions \cite[Theorem 2.6.2]{Dem09}. This shows the injectness of $\eta^\ast.$ Suppose now that $F\in\mathcal{O}(\check{\Omega}).$ Since $\eta: \check{\Omega}\setminus \eta^{-1}(E)\rightarrow \Omega\setminus E$ is biholomorphic, it follows that
$F\circ \eta^{-1}\in\mathcal{O}(\Omega\setminus E).$ Note that $\text{dim}_{\mathbb{C}} \hspace{0.1em} E\leq\text{dim}_{\mathbb{C}} \hspace{0.1em} V-1\leq \text{dim}_{\mathbb{C}} \hspace{0.1em} \Omega-2,$ the Riemann extension theorem on complex manifolds \cite[Section 7.1.3]{GR84} implies that there exists an unique $\hat{F}\in\mathcal{O}(\Omega)$ such that $\hat{F}=F\circ \eta^{-1}$ on $\Omega\setminus E.$ Clearly $F=\hat{F}\circ \eta$ on $\Omega$ by the uniqueness theorem. Then the bijection is proved. We now prove that  $ \eta^\ast:\mathcal{O}^{s}(\overline{\Omega})\rightarrow \mathcal{O}^{s}(\overline{\check{\Omega}})$ with conanical norms is a unitary operator.
By the above, it suffices to prove that $\eta^\ast:\mathcal{O}^{s}(\overline{\Omega})\rightarrow\mathcal{O}^{s}(\overline{\check{\Omega}})$ is an isometry with canonical norms for $s\in\mathbb{R}.$  Combined with (\ref{norms}) and (\ref{normss}), it is enough to prove that $ \eta^\ast:\mathcal{O}^{s}(\partial{\Omega}) \rightarrow  \mathcal{O}^{s}(\partial{\check{\Omega}})$ is an isometry with canonical norms for $s\in\mathbb{R}.$ For the case $s<0,$ this is immidiate from Lemma \ref{luni} and the definition of the conanical norms. The case $s=0$ is trivial.  It remains to prove the case $s>0.$ It follows from Lemma \ref{luni} that $\check{\Lambda}_0=\eta^\ast\Lambda_0(\eta^\ast)^{-1},$ namely $\check{\Lambda}_0 $ and $\Lambda_0$ are unitary. Thus  $\check{\Lambda}_0^{-1}=\eta^\ast\Lambda_0^{-1}(\eta^\ast)^{-1}.$  And hence, the strongly continuous semigroups $\{e^{\sqrt{-1} t\check{\Lambda}_0^{-1}}\}_{t\in\mathbb{R}}$ and $\{e^{\sqrt{-1} t{\Lambda}_0^{-1}}\}_{t\in\mathbb{R}}$ are unitary, where $\{e^{\sqrt{-1} t\check{\Lambda}_0^{-1}}\}_{t\in\mathbb{R}},  \{e^{\sqrt{-1} t{\Lambda}_0^{-1}}\}_{t\in\mathbb{R}}$  are uniquely determined by the skew-adjoint generators ${\sqrt{-1} \check{\Lambda}_0^{-1}}, {\sqrt{-1} {\Lambda}_0^{-1}},$ respectively. 
Since $\check{\Lambda}_0^{-1}, {\Lambda}_0^{-1}$  have order $1$ and $F_{-2s-1}\in S^{-2s-1}(\mathbb{R})$ for $s>0$ by the construction in Section 3, it follows from the functional calculus of pseudodifferential operators \cite[Theorem XII.1.3]{Tay81} that 
\begin{equation}\begin{split}\label{} 
 \check{B}_{-2s-1}&=F_{-2s-1} (\check{\Lambda}_0^{-1})\\
 &=\int_{\mathbb{R}}\mathcal{F}({F}_{-2s-1})(t)e^{\sqrt{-1} t\check{\Lambda}_0^{-1}}dt\\
 & =\int_{\mathbb{R}}\mathcal{F}({F}_{-2s-1})(t)\eta^\ast e^{\sqrt{-1} t{\Lambda}_0^{-1}}(\eta^\ast)^{-1}dt\\
 &=\eta^\ast\int_{\mathbb{R}}\mathcal{F}({F}_{-2s-1})(t) e^{\sqrt{-1} t{\Lambda}_0^{-1}}dt (\eta^\ast)^{-1}\\
 &=\eta^\ast F_{-2s-1} (\Lambda_0^{-1}) (\eta^\ast)^{-1}\\
&=\eta^\ast {B}_{-2s-1} (\eta^\ast)^{-1},\\
\end{split}
 \end{equation}
 where $\mathcal{F}$ is the  Fourier transform on  $L^1(\mathbb{R})$ as defined in (\ref{fouri}).
Thus $$\langle \check{B}_{-2s-1}\eta^\ast u,  \eta^\ast v \rangle_{\mathcal{O}^{s}(\partial{\check{\Omega}})}= \langle {B}_{-2s-1} u,  v \rangle_{\mathcal{O}^{s}(\partial{\Omega})} $$ for all $u,v\in\mathcal{O}^{s}(\partial{\Omega}),$ which implies that $ \eta^\ast:\mathcal{O}^{s}(\partial{\Omega}) \rightarrow  \mathcal{O}^{s}(\partial{\check{\Omega}})$ is an isometry with canonical norms for $s\in\mathbb{R}.$ This completes the proof.  \end{proof}

 \begin{rem}  With canonical norms as above, we have the following  commutative diagram with unitary operators
 \begin{displaymath} \xymatrix@R+1.3em{\mathcal{O}^{s+\frac{1}{2}}(\overline{\Omega}) \ar[r]^<(.2){\eta^\ast}& \mathcal{O}^{s+\frac{1}{2}}(\overline{\check{\Omega}}) \\
\mathcal{O}^{s}(\partial\Omega) \ar[r]^<(.25){\eta^\ast}\ar[u]^{K}&\mathcal{O}^{s}(\partial\check{\Omega})\ar[u]_{\check{K}}.\\}
\end{displaymath}
\end{rem}

Recall that $M_V=\{f\in \mathcal{O}^{s}({\overline{\Omega}},{w}^d) :f|_{V\cap {\Omega}}=0\}$ is a submodule of $\mathcal{O}^{s}(\overline{\Omega},{w}^d) .$
Denote by $M_{V^\prime}$ the subspace $M_{V^\prime}:=\{f\in \mathcal{O}^{s}(\overline{\check{\Omega}},\check{w}^d) :f|_{V^\prime\cap \check{\Omega}}=\bm{0} \},$ which is a submodule of $\mathcal{O}^{s}(\overline{\check{\Omega}},\check{w}^d) .$ Then the following corollary  is immediate from Proposition \ref{blh}.

\begin{cor}\label{inet}  (1)  The  permissible blowing-up $\eta:\check{\Omega}\rightarrow{\Omega}$  introduces a unitary operator $$ \eta^\ast:M_V\rightarrow  M_{V^\prime}, \quad f\longmapsto {\eta}^\ast(f)=f\circ\eta,$$
for all $s\in\mathbb{R}.$

(2) The  permissible blowing-up $\eta:\check{\Omega}\rightarrow{\Omega}$  introduces a unitary operator $$ \tilde{\eta}:M_V^\bot \rightarrow  M_{V^\prime}^\bot, \quad f\longmapsto \tilde{\eta}(f)=(P_{M_{V^\prime}^\bot}\circ\eta^\ast) f,$$
for all $s\in\mathbb{R}.$
\end{cor}

For each $f\in C^\infty(\overline{\Omega}),$ we denote $S^{\Omega}_{f}=P_{M_V^\bot}M_f$ by the compression of the multiplication operator $M_f,$ which is also called the truncated Toeplitz operator with symbol $f.$ We will write $S_{f}$ instead of $S^{\Omega}_{f}$ for short if $\Omega$ is unambiguous in the context.  Similarly, we can define the truncated Toeplitz operator on $\check{\Omega}.$ The following lemma gives the transition formula of truncated Toeplitz operators of holomorphic symbols under the blowing-up.
\begin{lem}\label{traf} If $f\in \mathcal{O}(\overline{\Omega}),$ then \begin{equation}\label{tfor}  \tilde{\eta}\circ S_f^{\Omega}\circ \tilde{\eta}^\ast=S_{\eta^\ast(f)}^{\check{\Omega}}.\end{equation}
\end{lem}
\begin{proof} For every $g\in M_{V^\prime}^\bot,$ it follows from Corollary \ref{inet} (2) that there exists a unique $h\in M_V^\bot$ satisfying $$\tilde{\eta}(h)=P_{M_{V^\prime}^\bot}(h\circ \eta) =(P_{M_{V^\prime}^\bot}\circ\eta^\ast)h=g.$$
Note that $ f\in \mathcal{O}(\overline{\Omega}),$ it follows from  Corollary \ref{inet}(1) that $P_{M_{V^\prime}^\bot}(\eta^\ast(P_{M_V}(fh)))=0$ and $P_{M_{V^\prime}^\bot}(\eta^\ast(f)P_{M_{V^\prime}}\eta^\ast(h))=0.$
Thus \begin{equation}\begin{split}\label{}
( \tilde{\eta}\circ S_f^{\Omega}\circ(\tilde{\eta})^\ast)g
&=\tilde{\eta}(P_{M_V^\bot}(fh))\\  \notag
&=\tilde{\eta}(fh-P_{M_V}(fh))\\
&=P_{M_{V^\prime}^\bot}(\eta^\ast(f)\eta^\ast(h))-P_{M_{V^\prime}^\bot}(\eta^\ast(P_{M_V}(fh)))\\
&=P_{M_{V^\prime}^\bot}(\eta^\ast(f)P_{M_{V^\prime}^\bot}\eta^\ast(h))+P_{M_{V^\prime}^\bot}(\eta^\ast(f)P_{M_{V^\prime}}\eta^\ast(h))\\
&=P_{M_{V^\prime}^\bot}(\eta^\ast(f)g)\\
&= S_{\eta^\ast(f)}^{\check{\Omega}}g.
\end{split}
 \end{equation}
It deduces that $ \tilde{\eta}\circ S_f^{\Omega}\circ(\tilde{\eta})^\ast=S_{\eta^\ast(f)}^{\check{\Omega}}.$
\end{proof}

 \begin{prop}\label{pnoc} Let $(W,\Omega)$ be a strongly pseudoconvex finite manifold with Property (S)  and $V$ be a compact smooth analytic subvariety in ${\Omega}.$ Then the Bergman-Sobolev quotient submodule  $M_V^\bot\subset \mathcal{O}^{s}(\overline{\Omega})$ is  unitary equivalent to $\mathbb{C}^m$ for some nonnegative integer $m.$ Hence,  $M_V^\bot\subset \mathcal{O}^{s}(\overline{\Omega})$ is $p$-essentially normal  for all $p>0.$
 \end{prop}
\begin{proof}  Since $V\subset\Omega$ is smooth and compact, it follows from the global decomposition theorem  for analytic subvariety that there exist at most  finitely many  mutually disjoint compact complex submanifolds $V_1,\cdots, V_m$ of $\Omega$ satisfying  $$V=\bigcup_{i=1}^m V_i.$$ 
Since $(W,\Omega)$ has Property (S), there exists a proper modification  $\iota: W\rightarrow\breve{W}$ for a Stein manifold  $\breve{W}$ such that $\iota$ is biholomorphic on a neighborhood of $\partial \Omega.$ The properness of $\iota$ implies that $\iota(V_1),\cdots,\iota(V_m)$ are compact analytic subvarieties in $\breve{W}.$  But  $\breve{W}$ is Stein, and so  $\iota(V_1),\cdots,\iota(V_m)$ are discrete
points, denoted by $p_1,\cdots,p_m,$ respctively. Denote by $\breve{\Omega}=\iota(\Omega),$ which is a relatively compact smooth strongly pseudoconvex in the Stein manifold $\breve{W}.$ Let $\iota^\ast:\mathcal{O}(\breve{\Omega})\rightarrow\mathcal{O}(\Omega), \iota^\ast(f)=f\circ\iota.$ Similar to Lemma \ref{blh}, one can prove that $\iota^\ast:\mathcal{O}^s(\overline{\breve{\Omega}})\rightarrow\mathcal{O}^s(\overline{\Omega})$ is unitary with suitable norms for each $s\in\mathbb{R}$ and $M_V=\iota^\ast M_{\iota(V)} (\iota^\ast )^{-1}.$ And hence $M_V^\bot=\iota^\ast M_{\iota(V)}^\bot (\iota^\ast )^{-1}.$ On the other hand, it follows from Proposition \ref{rpdu} that $\mathcal{O}^s(\overline{\breve{\Omega}})$ is a reproducing kernel Hilbert space. Let $K_{s,w}(z):=K_s(z,w)$ be the  reproducing kernel of $\mathcal{O}^s(\overline{\breve{\Omega}}).$ Then $M_{\iota(V)}^\bot =\overline{\text{Span}}\{K_{s,p_1},\cdots,K_{s,p_m}\}$  has finite dimension. Thus the quotient submodule $M_V^\bot\subset \mathcal{O}^{s}(\overline{\Omega})$ is unitary equivalent to $\mathbb{C}^m$ by the orthogonal projection theorem. It is evident that $[S_f,S_g^\ast]\in \mathcal{L}^p (M_V^\bot )$ for all $p>0,$ where $f,g\in \mathcal{O}(\overline{\Omega})$ are arbitrary. It leads to the desired result.
\end{proof}

We are now in a position to prove the main result Theorem \ref{main}. 

{\noindent{\bf{Proof of  Theorem \ref{main}.}}
(1) Since $\Omega\subset \tilde{\Omega}\Subset W$ and the analytic subvariety $V\subset \tilde{\Omega}$ transversely intersects $\partial\Omega,$ by Hironaka's embedded desingularization  theorem \cite[Theorem 6]{AHV18}, there exists a proper holomorphic map $\pi:\hat{\tilde{\Omega}}\rightarrow \tilde{\Omega}$ which is composed  by a finite sequence of permissible blowing-up $\pi_i: \tilde{\Omega}_{i}\rightarrow\tilde{\Omega}_{i-1}$  along with smooth centers (submanifolds) $E_{i-1}\subset \tilde{\Omega}_{i-1}, i=0,1,\cdots,r,$ i.e.,
$$\pi: \hat{\tilde{\Omega}}=\tilde{\Omega}_{r}\overset{\pi_r}{\longrightarrow} \tilde{\Omega}_{r-1}{\longrightarrow} \cdots\longrightarrow\tilde{\Omega}_{1}\overset{\pi_{1}}{\longrightarrow} \tilde{\Omega}_{0}=\tilde{\Omega},$$
$\pi=\pi_1\circ\cdots\circ\pi_r, $ and smooth subvariety $V_r\subset \tilde{\Omega}_{r}$ such that
\begin{enumerate}
\item $\pi_i$ is the blowing-up of the complex manifold $\tilde{\Omega}_{i-1}$ along with the smooth centre $E_{i-1}$ contained in $\text{Sing}(V_{i-1}),$ where $V_i$ is the strict transform of $V_{i-1}$ under $\pi_i,$ and $V_i$  $i=0,\cdots,r,$ and $V_0=V;$
\item  $E^r:={\pi}^{-1}(\text{Sing}(V))$ is a normal crossing divisor in $\tilde{\Omega}_{r};$
\item  $ \pi:\tilde{\Omega}_{r}\setminus E^r\rightarrow \Omega\setminus\text{Sing}(V)$ and $\pi: V_{r}\setminus E^r\rightarrow  V\setminus\text{Sing}(V) $ are both biholomorphic;
\item $ E^r$ and $V_r$ has only normal crossing.
\end{enumerate}
Let $\Omega_{0}=\Omega,\Omega_{i}=(\pi_1\circ\cdots\circ\pi_i)^{-1}(\Omega_0),i=1,\cdots,r.$ Clearly, $\Omega_{i}\subset  \tilde{\Omega}_i,i=0,\cdots,r.$ Since $\pi_i$ is biholomorphic on some neighborhood of the boundary $\partial\Omega_i,$ it follows that $V_i$ transversely intersects  $\partial\Omega_i,i=0,\cdots,r.$ Note that $E_i\subset \Omega_i, i=0,\cdots,r,$ it follows from the definition of blowing-up that the restriction $\pi_i|_{\Omega_i}:\Omega_i\rightarrow \Omega_{i-1}$ is the blowing-up along the smooth centre $E_{i-1},i=0,\cdots,r.$
And hence the proper map $$\pi|_{\Omega_r}:=\pi_1|_{\Omega_1}\circ\cdots \circ\pi_r|_{\Omega_r}: \Omega_r\rightarrow\Omega$$ is an embedded desingularization for subvariety $V\cap\Omega\subset \Omega.$
By rapidly using of  the transition formula (\ref{tfor}), we have $$\widetilde{\pi|_{\Omega_r}}\circ S_f^{\Omega}\circ \widetilde{\pi|_{\Omega_r}}^\ast=S^{\Omega_r}_{\pi|_{\Omega_r}^\ast(f)}$$ for all $f\in \mathcal{O}(\overline{\Omega}),$ where $\widetilde{\pi|_{\Omega_r}}=\widetilde{\pi_r|_{\Omega_r}}\circ\cdots\circ\widetilde{\pi_1|_{\Omega_1}}$ is  a composition of  finitely many unitary operators and $\pi|_{\Omega_r}^\ast$ is the pullback of $\pi|_{\Omega_r}.$ Thus for arbitrary $f,g\in \mathcal{O}(\overline{\Omega}),$ it follows that \begin{equation}\label{pisfs}\widetilde{\pi|_{\Omega_r}} [S_f^{\Omega},{S_g^{\Omega}}^\ast] \widetilde{\pi|_{\Omega_r}}^\ast=[S^{\Omega_r}_{\pi|_{\Omega_r}^\ast(f)},{S^{\Omega_r}_{\pi|_{\Omega_r}^\ast(g)}}^\ast].\end{equation} This implies that $[S_f^{\Omega},{S_g^{\Omega}}^\ast] \in \mathcal{L}^p(M_{V}^\bot)$ if and only if $ [S^{\Omega_r}_{\pi|_{\Omega_r}^\ast(f)},{S^{\Omega_r}_{\pi|_{\Omega_r}^\ast(g)}}^\ast]\in\mathcal{L}^p(M_{V_r}^\bot)$ for each $p>0.$ On the other hand, $V_r\cap\Omega_r \subset \Omega_r $ is a smooth subvariety,  this implies from the global decomposition theorem for analytic subvariety that the smooth analytic subvariety $V_r\cap\Omega_r $ has the following decomposition \begin{equation}\label{vdop}V_r\cap\Omega_r =(V_r\cap\Omega_r )^+\cup (V_r\cap\Omega_r )^-,\end{equation} where $(V_r\cap\Omega_r )^-$ is the at most finitely many unions of compact irreducible components and  $(V_r\cap\Omega_r )^+$ is the finitely many unions of noncompact irreducible components. Clearly, $(V_r\cap\Omega_r )^+$ is a noncompact smooth analytic subvariety of $\Omega_r$ which intersects $ \partial \Omega_r $ transversally,  and $ (V_r\cap\Omega_r )^-$  is a compact smooth analytic subvariety of $\Omega_r.$ From the decomposition (\ref{vdop}), we see that $$M_{V_r}^\bot=M_{(V_r\cap\Omega_r )^+}^\bot\oplus M_{(V_r\cap\Omega_r )^-}^\bot.$$
Combining this with Proposition \ref{pnoc}, this implies that $M_{(V_r\cap\Omega_r )^-}^\bot$ is  unitary equivalent to $\mathbb{C}^m$ for some nonnegative integer $m$ and $$P_{M_{(V_r\cap\Omega_r )^-}^\bot}[S^{\Omega_r}_{\pi|_{\Omega_r}^\ast(f)},{S^{\Omega_r}_{\pi|_{\Omega_r}^\ast(g)}}^\ast]P_{M_{(V_r\cap\Omega_r )^-}^\bot}\in \mathcal{L}^p(M_{(V_r\cap\Omega_r )^-}^\bot)$$ for all $p>0.$ On the other hand, when $(V_r\cap\Omega_r )^+\neq \emptyset,$
it follows from Proposition \ref{pnor} that $$P_{M_{(V_r\cap\Omega_r )^+}^\bot}[S^{\Omega_r}_{\pi|_{\Omega_r}^\ast(f)},{S^{\Omega_r}_{\pi|_{\Omega_r}^\ast(g)}}^\ast]P_{M_{(V_r\cap\Omega_r )^+}^\bot}\in \mathcal{L}^p(M_{(V_r\cap\Omega_r )^+}^\bot)$$ for all $$p>\text{dim}_{\mathbb{C}} \hspace{0.1em} (V_r\cap\Omega_r )^+=\text{dim}_{\mathbb{C}} \hspace{0.1em} (V_{r-1}\cap\Omega_{r-1} )^+=\cdots=\text{dim}_{\mathbb{C}} \hspace{0.1em} (V\cap\Omega )^+.$$ 
Thus $[S_f^{\Omega},{S_g^{\Omega}}^\ast] \in \mathcal{L}^p(M_{V}^\bot)$ for all $p>\text{dim}_{\mathbb{C}} \hspace{0.1em} (V\cap\Omega )^+,$  if $(V\cap\Omega )^+\neq \emptyset.$ 

(2) It follows from \cite[Proposition 1.1 (c)]{HeH75} that  \begin{equation}\label{comee} [S_{f_1},S_{g_1}^*,\cdots,S_{f_l},S_{g_l}^*]=\sum_{\sigma,\tau\in{S}_l}\text{sgn}(\sigma\tau)[S_{f_{\sigma(1)}},S_{g_{\tau(1)}}^*]\cdots [S_{f_{\sigma(l)}},S_{g_{\tau(l)}}^*].\notag\\ \end{equation}
Applying (\ref{pisfs})  gives \begin{equation}\begin{split}\label{swgf} &[S_{f_1},S_{g_1}^*,\cdots,S_{f_l},S_{g_l}^*] \\
&=   \widetilde{\pi|_{\Omega_r}}^\ast [ S^{\Omega_r}_{\pi|_{\Omega_r}^\ast(f_1)},{S^{\Omega_r}_{\pi|_{\Omega_r}^\ast(g_1)}}^\ast,\cdots, S^{\Omega_r}_{\pi|_{\Omega_r}^\ast(f_1)},{S^{\Omega_r}_{\pi|_{\Omega_r}^\ast(g_1)}}^\ast]  \widetilde{\pi|_{\Omega_r}}  \\
&=  \widetilde{\pi|_{\Omega_r}}^\ast\sum_{\sigma,\tau\in{S}_l}\text{sgn}(\sigma\tau)  [ S^{\Omega_r}_{\pi|_{\Omega_r}^\ast(f_1)},{S^{\Omega_r}_{\pi|_{\Omega_r}^\ast(g_1)}}^\ast]  \cdots [S^{\Omega_r}_{\pi|_{\Omega_r}^\ast(f_{\sigma(l)})},{S^{\Omega_r}_{\pi|_{\Omega_r}^\ast(g_{\tau(l)})}}^\ast] \widetilde{\pi|_{\Omega_r}} \\
\end{split}\end{equation} 
Note that  the subvariety $V_r\subset \tilde{\Omega}_{r}$ is smooth. Combining (\ref{swgf})  with (\ref{Sfg}),  it follows from  (P3) and (P5)  that $[S_{f_1},S_{g_1}^*,\cdots,S_{f_l},S_{g_l}^*] $ can be written as the following form 
$$[S_{f_1},S_{g_1}^*,\cdots,S_{f_l},S_{g_l}^*] =  \widetilde{\pi|_{\Omega_r}}^\ast(\tau^\ast R_l \tau  +C_l )\widetilde{\pi|_{\Omega_r}},$$ where $R_l$ is a diagonal operator matrix of generalized Toeplitz operators with order no more than $-(l+1)$ and $C_l$ is an operator belonging to $\mathcal{L}^p$ for all $p>0.$ Thus $[S_{f_1},S_{g_1}^*,\cdots,S_{f_l},S_{g_l}^*]\in \mathcal{L}^p$ for all  $p> \frac{1}{l+1} \hspace{0.1em} {\text{dim}_{\mathbb{C}}\hspace{0.1em} (V\cap \Omega)}^+$ by Lemma \ref{lpq}. In particular,  we have $[S_{f_1},S_{g_1}^*,\cdots,S_{f_l},S_{g_l}^*]\in \mathcal{L}^1$ if $l\geq \text{dim}_{\mathbb{C}}\hspace{0.1em} (V\cap \Omega)^+.$ 

Suppose now $l>\text{dim}_{\mathbb{C}}\hspace{0.1em} (V\cap \Omega)^+.$ By direct calculation we obtain 
\begin{equation}\begin{split}\label{ssfgs}
[S_{f_1},S_{g_1}^*,\cdots,S_{f_l},S_{g_l}^*]
&=\sum_{\sigma,\tau\in{S}_l}\text{sgn}(\sigma\tau)[S_{f_{\sigma(1)}},S_{g_{\tau(1)}}^*]\cdots [S_{f_{\sigma(l)}},S_{g_{\tau(l)}}^*]\\
&= \sum_{i,j=1}^{l} (-1)^{i+j}\big( [S_{f_{i}},S_{g_j}^* [S_{f_{1}},S_{g_{1}}^*,\cdots,S_{f_{l}},S_{g_{l}}^\ast]^{\hat{i},\hat{j}^\ast}\\
& \quad \quad -  S_{g_{j}}^*[S_{f_{i}}, [S_{f_{1}},S_{g_{1}}^*,\cdots,S_{f_{l}},S_{g_{l}}^\ast]^{\hat{i},\hat{j}^\ast}]\big),\\
\end{split}
 \end{equation}
where the symbol $[S_{f_{1}},S_{g_{1}}^*,\cdots,S_{f_{l}},S_{g_{l}}^\ast]^{\hat{i},\hat{j}^\ast}$ means that the operators 
$S_{f_i},S_{g_j}^\ast$ are deleted from the antisymmetric sum $[S_{f_1},S_{g_1}^*,\cdots,S_{f_l},S_{g_l}^*]$ of $2l$ operators.
From the above, we see that $[S_{f_{1}},S_{g_{1}}^*,\cdots,S_{f_{l}},S_{g_{l}}^\ast]^{\hat{i},\hat{j}^\ast}\in\mathcal{L}^1$ when $l>\text{dim}_{\mathbb{C}}\hspace{0.1em} (V\cap \Omega)^+.$  By \cite[Lemma 1.3]{HeH75}, it follows that $$\text{Tr}\hspace{0.1em}[S_{f_{i}}+S_{f_{i}}^\ast, [S_{f_{1}},S_{g_{1}}^*,\cdots,S_{f_{l}},S_{g_{l}}^\ast]^{\hat{i},\hat{j}^\ast}]=\text{Tr}\hspace{0.1em}[\sqrt{-1}(S_{f_{i}}-S_{f_{i}}^\ast), [S_{f_{1}},S_{g_{1}}^*,\cdots,S_{f_{l}},S_{g_{l}}^\ast]^{\hat{i},\hat{j}^\ast}]=0.$$ Thus $\text{Tr}\hspace{0.1em}[S_{f_{i}}, [S_{f_{1}},S_{g_{1}}^*,\cdots,S_{f_{l}},S_{g_{l}}^\ast]^{\hat{i},\hat{j}^\ast}]=0.$ Similarly,  we have $$\text{Tr}\hspace{0.1em}[S_{f_{i}},S_{g_j}^* [S_{f_{1}},S_{g_{1}}^*,\cdots,S_{f_{l}},S_{g_{l}}^\ast]^{\hat{i},\hat{j}^\ast}]=0.$$ Combined with (\ref{ssfgs}), we deduce that  $[S_{f_1},S_{g_1}^*,\cdots,S_{f_l},S_{g_l}^*]$ has zero trace. This completes the proof.
\qed

\section{Applications}
This section is devoted to immediate applications of the main theorem in the $K$-homology theory and geometric invariant theory.

\subsection{The $K$-homology theory} In this subsection we pay attention to the case where $\Omega$ is a bounded  smooth strongly pseudoconvex domain in  $\mathbb{C}^d,$  and $V$ is a proper  analytic subvariety in a neighborhood of $\overline{\Omega}$ with possible only isolate singularities inside $\Omega,$ which transversely intersects the  smooth  boundary $\partial \Omega.$
It follows from Lemma \ref{phda} that the Bergman-Sobolev space $\mathcal{O}^0(\overline{\Omega})$ coincides with the usual Bergman space. It follows Corrollary \ref{corA} that the quotient submodule $M_V^\bot\subset\mathcal{O}^0(\overline{\Omega})$ is $p$-essentially normal for $p>\text{dim}_{\mathbb{C}} \hspace{0.1em} (V\cap\Omega).$ Let 
 $\mathcal{J}(M_V^\bot)$ be the $C^\ast$-algebra generated by the operators defined by module multiplication on $M_V^\bot$ and $\mathcal{C}(M_V^\bot)$  be the closed two-side ideal generated by  commutators in $\mathcal{J}(M_V^\bot).$ Thus $\mathcal{J}(M_V^\bot)$ consists of compact operators and hence $(\mathcal{J}(M_V^\bot)+ \mathcal{K}(M_V^\bot))/\mathcal{K}(M_V^\bot)$ is a commutative $C^\ast$-algebra, where $\mathcal{K}(M_V^\bot)$ is the compact operator ideal. Denote $X_{M_V^\bot}$ by the maximal ideal space, which is a compact Hausdorff space.    The Gelfand-Naimark duality  gives the isometrical isomorphism between $ (\mathcal{J}(M_V^\bot)+ \mathcal{K}(M_V^\bot))/\mathcal{K}(M_V^\bot) $ and $C(X_{M_V^\bot}),$ which produces the short exact sequence of  $C^\ast$-algebras
  \begin{equation}\label{exat} 0\rightarrow \mathcal{K}(M_V^\bot)\rightarrow  \mathcal{J}(M_V^\bot)+ \mathcal{K}(M_V^\bot)\rightarrow C(X_{M_V^\bot})\rightarrow 0.\end{equation}

By the BDF-theory in the $C^\ast$-algebra, the short exact sequence (\ref{exat}) represents an equivalence class denoted by $[M_V^\bot]$ in the odd  $K$-homology group $K_1(X_{M_V^\bot}),$ we refer the reader to \cite[Chapter 5]{Upm96} for more details about $K$-theory. Combined with \cite[Theorem 3]{Dou06A}, it follows the following result.

\begin{prop} \label{Kth} Let $\Omega\subset \mathbb{C}^d$ be a bounded  smooth strongly pseudoconvex domain and $V$ be a proper  analytic subvariety in a neighborhood of $\overline{\Omega}$ with possible only isolate singularities inside $\Omega,$ which transversely intersects the  smooth  boundary $\partial \Omega.$   Then the homology  class $[M_V^\bot]$ represented by the Bergman quotient submodule is in $K_1(V\cap \partial \Omega).$
\end{prop}

In the case where $\Omega=\mathbb{B}^d$ and $ V$ is a nice algebraic variety, the homology class $[M_V^\bot]$ can be explicitly constructed by the resolution of the Hilbert module, see \cite{DJTY18} and reference therein.

\subsection{The geometric invariant theory}
Let  $H^2_d$ be the Drury-Arveson space on the unit ball $\mathbb{B}^d,$ then $H^2_d$ and $\mathcal{O}^{d/2}(\overline{\mathbb{B}^d})$ are coincide by Lemma \ref{phda}. For each $r\geq1,$   the coordinate multiplier $z_i,i=1,\cdots,d,$ determine a natural multiplicative operator $d$-tuple on $H^2_d\otimes\mathbb{C}^r,$   denoted by  $M= (M_1,\cdots, M_d)$ without ambiguity,
which introduces a natural graded module structure of finite rank on  $H^2_d\otimes\mathbb{C}^r$  over the polynomial algebra $\mathbb{C}[\bm{z}].$ Then the $d$-tuple $M$ is a $d$-contraction in the sense of $$ \Vert M_1f_1+\cdots+M_df_d\Vert^2 \leq  \Vert f_1\Vert^2+\cdots+ \Vert f_1\Vert^2$$  for all $f_1,\cdots,f_d\in H^2 \otimes\mathbb{C}^r.$

 Let $\delta_T$ be the Dirac operator associated an arbitrary $d$-contraction $T$ on $H^2_d\otimes\mathbb{C}^r,$  which is a self-adjoint operator acting on the Hilbert space  $(H^2_d\otimes\mathbb{C}^r)\otimes \Lambda \mathbb{C}^d,$ where  $\Lambda \mathbb{C}^d$ is the exterior algebra of  $\mathbb{C}^d $ which is still a Hilbert space in the natural sense, see more details in \cite{Arv05}. The Dirac operator of $T$ is uniquely determined by $T$ up to isomorphism.   By the $\mathbb{Z}_2$-graded structure, it follows  the  orthogonal decomposition $$ (H^2_d\otimes\mathbb{C}^r)\otimes \Lambda \mathbb{C}^d=H_+\oplus H_-,$$ where vectors in $H_+$ (resp. $H_-$ ) are called even (resp. odd). Then the  Dirac operator $\delta_T$ has the 
following  matrix representation 
$$\delta_T=\left[\begin{array}{cc}0 & \delta_+^\ast \\ \delta_+ & 0\end{array}\right]$$ such that  $ \delta_+ H_+\subset H_-.$ The  the  Fredholm index of the $d$-tuple  $T$ is defined to be  $$\text{Ind}\hspace{0.1em} (T)= \dim \hspace{0.1em}\text{Ker}\hspace{0.1em} \delta_+-\dim\hspace{0.1em} \text{Ker}\hspace{0.1em} \delta_+^\ast,$$ see details in \cite[Section 1]{Arv05}. The $d$-tuple $T$ is said to be Fredholm if its Dirac operator $\delta_T$ is a Fredholm operator; in this case, its index is an integer. 
 
 On the other hand, for every finite rank $d$-contraction $T$ on $H^2_d\otimes\mathbb{C}^r,$ it is possible to define a real number $K(T)$ in the interval $[0, \text{rank} \hspace{0.1em}T]$ defined as the integral of the trace of a certain matrix-valued function over the unit sphere $\partial \mathbb{B}^d,$ which is a geometric invariant called the curvature invariant of the $d$-tuple $T.$  The following result, which is related to the geometric invariant and the Fredholm index of the operator tuple, extends \cite[Proposition 7.10]{Arv00} and provides a positive answer to \cite[Problem D]{Arv07} for the compression of the $d$-shift of rank $1.$
   
  \begin{prop}\label{ginv}  Suppose that   $V$ is a  proper analytic  subvariety  in some neighborhood of $\overline{\mathbb{B}^d}.$ If $V$ is smooth on $\mathbb{B}^d$  and  intersects transversally with  $\partial \mathbb{B}^d,$ then the following hold.
  
(1) The  compression   $S$ of the $d$-tuple of the respective coordinate operators on the quotient submodule  $M_V^\bot\subset H^2_d$ is Fredholm and 
 \begin{equation}\label{gb}\text{Ind}\hspace{0.1em} (S)=0.\notag\\ \end{equation} 

 (2)  Furthermore, if $ V$ is a  homogeneous algebraic variety in $\mathbb{C}^d$ with only possible isolate singularity of the origin, then the  compression   $S$ on the  graded quotient submodule  $M_V^\bot\subset H^2_d$ is Fredholm and 
 \begin{equation}\label{gbc}K(S)=\text{Ind}\hspace{0.1em} (S)=0.\notag\\ \end{equation} 
 \end{prop}
\begin{proof} (1) Denote $S=(S_1,\cdots,S_d).$
  Corollary \ref{corA} (1) implies that  $[S_i,S_i^\ast]\in \mathcal{L}^p,$ for all $p> \text{dim}_{\mathbb{C}}\hspace{0.1em} (V\cap \mathbb{B}^d), i=1,\cdots,d.$  
Since the defect operator $1-\sum_{i=1}^d  S_i S_i^\ast$ has finite rank by \cite[Lemma 2.8]{Arv98} and \cite[Proposition 4.1]{Arv05}, it follows that 
$$ 1-\sum_{i=1}^d S_i^\ast S_i=1-\sum_{i=1}^d  S_i S_i^\ast+\sum_{i=1}^d[S_i,S_i^\ast] \in \mathcal{L}^p,$$ for all $p> \text{dim}_{\mathbb{C}}\hspace{0.1em} (V\cap \mathbb{B}^d) .$   Combined  with the condition that the subvariety  $V$ is proper,  it follows that  $ 1-\sum_{i=1}^d S_i^\ast S_i  \in \mathcal{L}^d.$  Then,   \cite[Proposition 8.1]{WaX23} implies that $$ \text{Ind}\hspace{0.1em} (S)=   \text{Tr}\hspace{0.1em}[S_{1},S_{1}^*,\cdots,S_{d},S_{d}^*].$$ Combing this with Corrollary \ref{corA} (2)  follows the index formula  $\text{Ind}\hspace{0.1em} (S)=0.$

(2) Since the quotient submodule  $M_V^\bot\subset H^2_d$ 
is pure contractive (see \cite[page 230]{Arv07}),  combined with  \cite[Theorem B]{Arv02} and \cite[Theorem 4.3]{Arv05},  it follows that 
 $$(-1)^d K(S)=\text{Ind}\hspace{0.1em} (S).$$ Then  Proposition \ref{ginv} (1) implies the desired identity. This completes the proof.
 \end{proof}

\hspace{0.9em}

{\noindent{\bf{Acknowledgements.}}

The author would like to thank Prof. Harald Upmeier for his insightful discussions and valuable suggestions. The author also wishes to thank Prof. Jingbo Xia for his illuminating discussions on the antisymmetric sum. The author was partially supported by the National Natural Science Foundation of China (12201571).

 \bibliographystyle{alpha}
 
\end{document}